\documentclass[a4paper]{amsart}
\usepackage{amsmath,amsthm,amssymb,latexsym,epic,bbm,comment,color}
\usepackage{graphicx,enumerate,stmaryrd}
\usepackage[all,2cell]{xy}
\xyoption{2cell}

\usepackage{tikz-cd}

\usepackage[all]{xy}
\usepackage[active]{srcltx}
\usepackage[parfill]{parskip}

\newcommand{\commentMR}[1]{{}} 

\newcommand{\newMR}{} 

\newcommand{\Z}{\mathbb Z}
\newcommand{\R}{\mathbb R}

\newcommand{\g}{\mathfrak{g}}
\newcommand{\p}{\mathfrak{p}}
\newcommand{\h}{\mathfrak{h}}
\newcommand{\aaa}{\mathfrak{a}}

\newcommand{\mcG}{\mathcal G}
\newcommand{\mcO}{\mathcal O}
\newcommand{\mcM}{\mathcal M}
\newcommand{\mcF}{\mathcal F}
\newcommand{\mcJ}{\mathcal J}
\newcommand{\mcA}{\mathcal A}
\newcommand{\und}{\underline}
\newcommand{\pa}{\partial}

\newcommand{\T}{\mathrm{T}}
\newcommand{\F}{\mathrm{F}}
\newcommand{\G}{\mathrm{G}}
\newcommand{\V}{\mathrm{V}}

\newcommand{\oT}{\mathbf{T}}
\newcommand{\gr}{\mathrm{gr}}
\newcommand{\red}{\mathrm{red}}
\newcommand{\dd}{\mathrm{d}} 
\newcommand{\pr}{\mathrm{pr}}
\newcommand{\Vb}{\mathrm{Vb}}
\newcommand{\At}{\mathrm{At}} 
\newcommand{\id}{\mathrm{id}}
\newcommand{\Euler}{\mathbbm{E}}
\newcommand{\Sec}{\mathrm{\Gamma}}

\newcommand{\wh}{\widehat}
\newcommand{\vareps}{\varepsilon}

\newcommand{\SLie}{\mathrm{sLieAlg}}
\newcommand{\GrLie}{\mathrm{grLieAlg}}
\newcommand{\DVB}{{\textbf{DVB }}}
\newcommand{\DVBs}{{\textbf{DVBs}}}

\newcommand{\rd}[1]{\left[{#1}\right]} 

\newcommand{\Lie}{\operatorname{Lie}}
\newcommand{\End}{\operatorname{End}}
\newcommand{\Hom}{\operatorname{Hom}}
\newcommand{\Der}{\operatorname{Der}}
\newcommand{\Aut}{\operatorname{Aut}}
\newcommand{\Sym}{\operatorname{Sym}}
\newcommand{\ad}{\operatorname{ad}}

\newtheorem{theorem}{Theorem}
\newtheorem{lemma}[theorem]{Lemma}
\newtheorem{definition}[theorem]{Definition}
\newtheorem{corollary}[theorem]{Corollary}
\newtheorem{proposition}[theorem]{Proposition}

\newtheorem{example}[theorem]{Example}

\newtheorem{remark}[theorem]{Remark} 

\begin{document}
\title[Graded covering of a supermanifold I.  The case of a Lie supergroup]
	{Graded covering of a supermanifold I. \\ The case of a Lie supergroup}
	
\dedicatory{Dedicated to the memory of Arkady Onishchik, 1933–2019}

\author{Mikolaj Rotkiewicz and Elizaveta Vishnyakova}

\begin{abstract}
We generalize the Donagi and Witten construction of a first obstruction class for splitting of a  supermanifold via differential operators using the theory of $n$-fold vector bundles and graded manifolds.  Applying the generalized Donagi--Witten construction we obtain a family of embeddings of the category of supermanifolds into the category of $n$-fold vector bundles and into the category of graded manifolds. This leads to a realization of any non-split supermanifold in terms of a collection of vector bundles {\newMR and some morphism between them.} 
Further we study the images of these embeddings into the category of graded manifolds in the case of a Lie supergroup and a Lie superalgebra. We show that these images satisfy universal property of a graded covering or a graded semicovering.

\end{abstract}

\maketitle


\section{Introduction}

This is the first part of a series of papers about $\Z$-graded coverings of a supermanifold. In this paper we investigate the case of a Lie supergroup and its Lie superalgebra. In \cite[Section 2]{Witten Atiyah classes} Donagi and Witten gave a description of the first obstruction class $\omega_2$ for a supermanifold to be split  via differential operators. More precisely, Donagi and Witten constructed an exact sequence of locally free sheaves with Atiyah class $\omega_2$. This is a very interesting and important observation that pure even vector bundles can keep certain information about the original (non-split) supermanifold. In the case a supermanifold has odd dimension $2$, the exact sequence of Donagi and Witten keeps the whole information of this supermanifold.

In this paper we generalize this idea. We show that the Donagi and Witten exact sequence corresponds to a  {\newMR double vector bundle with some additional structure, the whole information of which can be  reduced to a graded manifold of degree $2$. }
Further we
suggest to use  {\newMR a dual approach which leads to iterated differential forms, i.e.  functions on the iterated antitangent bundle (as it is explained below)} \commentMR{Are they iterated differential forms  in the sense of Vinogradov? I do not see you use this term any longer.} instead of differential operators, which  leads to a simplification of the Donagi and Witten construction and  suggests a natural generalization. For any (non-split) supermanifold   $\mcM$ of odd dimension $\leq n$ we construct a pure even geometric object, an $n$-fold vector bundle $\Vb_n(\mcM)$ which keeps the  whole information about the original supermanifold. Moreover the obtained $n$-fold vector bundle possesses additional symmetries. We interpret this in the following way: {\it the $n$-fold vector bundle arising from a supermanifold, can be reduced to a graded manifold of degree $n$}.  

Summing up we construct a family $\F_n$, where $n=2,3,\ldots \infty$, of functors from the category of supermanifolds of odd dimension $n$ to the category of $n$-fold vector bundles and to the category of graded manifolds of degree $n$, where the first functor $\F_2$ is a modification of the Donagi--Witten construction. Each functor $\F_n$ is a composition of four functors:  the $(n-1)$-iteratation of the  antitangent functor $\oT$, the functor split $\gr$ ({\newMR $\gr \mcM$ denotes the retract of the supermanifold $\mcM$})   \commentMR{The name ''functor split'' is not commonly used. Either say use retract or the retract which will be called the split functor}, which is defined in the category of supermanifolds, the functor parity change $\pi$, which is defined in the category of $n$-fold vector bundles \commentMR{I still see $\pi$ as the functor applied to a vector bundle: given an $n$-tuple VB we consider one of the vector bundle and take the parity reversion with respect to this VB structure. Namely, we change the parity of $\beta$ and so of $\alpha+\beta$ thus we consider VB with fibers of parities $\beta$ and $\alpha+\beta$.}
and another functor $\iota$, which we call inverse 
that was discovered independently in \cite{Grab} ({\newMR under the name the diagonalization functor}) and \cite{Vish}. We show that the functor $\F_n$ determines an embedding of the category of super\-manifolds of odd dimension $n$ (or smaller) into the category of graded manifolds of degree $n$. The functor $\F_{\infty}$ is a limit of functors $\F_n$. This functor defines an embedding of the category of supermanifolds into the category of  graded manifolds of any degree. This means that the graded manifold $\F_{\infty} (\mcM)$, the image of a (non-split) supermanifold  $\mcM$, contains the whole information about the original supermanofold $\mcM$. 

Moreover, for a Lie supergroup $\mcG$ we prove that $\F_{\infty} (\mcG)$ satisfies universal properties. These properties of $\F_{\infty} (\mcG)$ led use to introduce the notion of a graded covering of $\mcG$.
 That is: the covering space is unique up to isomorphism and every homomorphism from  a graded Lie supergroup to the Lie supergroup $\mcG$ can be lifted to $\F_{\infty} (\mcG)$. 
 We also introduce a notion of a graded semicovering of $\mcG$. This explains the meaning of other functors $\F_n$.   Our covering (semicovering) is in some sense a "local diffeomorphism" of Lie supergroups. Further we show that the Lie superalgebra of  $\F'_{\infty} (\g)$ is an example of a loop algebra construction by Allison, Berman, Faulkner, and Pianzola \cite{Allison}, see also \cite{Eld}. The loop algebra was used by several authors to investigate graded-simple (Lie) algebras. Therefore results of our paper establish a connection between Donagi--Witten observation and pure algebraic results. Our method leads to a notion of a graded covering and semicovering of any supermanifold as well. However due to technical difficulties of all proofs in this case we present the general theory of covering and semicovering spaces for any supermanifold in the second part of this research \cite{RV}. \commentMR{I see we have a graded covering of a supermanifold but we do not define it? I need to study it in more detail.}

Our results are especially interesting in the complex-analytic (and algebraic) category. The reason is the following. According to the Batchelor--Gawedzki Theorem any smooth supermanifold  is (non-canonically) split, that is its structure sheaf is isomorphic to the wedge power of a certain vector bundle. Therefore very often  we can study geometry of a split super\-manifold using geometry of vector bundles. This is not the case in the complex-analytic situation. The study of non-split super\-manifolds was initiated in \cite{Ber,Green}, where the first non-split supermanifold was described. Significant advances in this direction were achieved in the work of A.L.~Onishchik. Interest in this problem arose again after Donagi and Witten's papers \cite{Witten not projected,Witten Atiyah classes}, where they proved that the moduli space of super Riemann surfaces is non-split (or more generally not projected). To each (non-split) supermanifold $\mcM$ our functor $\F_n$ associates a graded manifold $\F_n(\mcM)$ and the image $\F_{n} (\mcM)$  keeps the whole information about the (non-split) supermanifold $\mcM$ for sufficiently large $n$.  Therefore we can study a non-split supermanifold in the category of graded manifolds using the tools of classical complex geometry due the fact that geometrically a graded manifold is a family of vector bundles. 

Our method studying a supermanifold is interesting in the case of a  split supermanifold as well, for instance if a split supermanifold possesses an additional structure. Any Lie supergroup $\mcG$ is {\newMR totally split, i.e. as a supermanifold it is a cartesian product of its underlying space $\mcG_0$ with a fixed odd vector space $\g_{\bar{1}}$ --- the odd part of the Lie superalgebra of $\mcG$. } 
However the supermanifold $\mcG$ has an additional structure which can be non-trivial: the Lie supergroup structure. We show that any graded Lie supergroup $\F_n(\mcG)$ contains the whole information about the original supergroup $\mcG$ for any $n$, non only for sufficiently large $n$.  
Moreover our result  is the first step toward a graded covering of a homogeneous supermanifold, which is very often non-split. For instance all (except of few exceptional cases) super-grassmannians and flag supermanifolds are non-split.

\textbf{Acknowledgements:} M.~R. was partially  supported by Capes/PrInt, UFMG. 
E.~V. was partially  supported by 
Coordenação de Aperfeiçoamento de Pessoal de Nível Superior - Brasil (CAPES) -- Finance Code
001, (Capes-Humboldt Research Fellowship)  and by Tomsk State University, Competitiveness Improvement Program. {\newMR M.~R. thanks  UFMG Department of Mathematics for its kind hospitality during
the realization of this work.}
The authors thank Peter Littelmann for his suggestion to investigate the case of a Lie supergroup and a Lie superalgebra.   This helped us to understand the meaning of the functor $\F_{\infty}$. We also thank 
Alexey Sharapov and Mikhail Borovoi for useful comments.

\section{Preliminaries}

\subsection{Lie superalgebras, graded Lie superalgebras and graded Lie superalgebras of type $\Delta$}

Throughout the paper we work over the field $\mathbb K=\mathbb C$ or $\mathbb R$. However the algebraic part of the paper holds true over any field of characteristic $0$.  We shall work in a supergeometry context. Our vector spaces  are $\mathbb Z_2$-graded, $V = V_{\bar 0} \oplus V_{\bar 1}$, elements of $V_{\bar i}\setminus \{0\}$ are called homogenous of \emph{parity} $\bar i\in \Z_2$. For $v \in V_{\bar i}$ we write $|v|=\bar i$.
Our main references on supergeomtery, in particular on Lie superalgebras, are  \cite{Kac,BLMS,Bern,Leites} and \cite{Var}.

\begin{definition}
	A Lie superalgebra is a $\mathbb Z_2$-graded  vector superspace 
$\mathfrak g = \mathfrak g_{\bar 0}\oplus \mathfrak g_{\bar 1}$ with  a graded bilinear operation $[\,, ]$ satisfying the following conditions:
\begin{equation}\label{ax:skew-symmetry}		
		 [x,y]=-(-1)^{|x||y|}[y,x],
\end{equation}
\begin{equation}\label{ax:Jacobi}
		[x, [y, z]]- (-1)^{|x||y|} [y, [x, z]] = [[x, y], z],
\end{equation}		
	where $x$, $y$, and $z$ are homogeneous elements. 
\end{definition}

 In more general setting, if $A$ is an abelian group, 
 an $A$-graded Lie superalgebra is a vector space $\g$ graded by an abelian group $A$, $\g = \oplus_{a\in A} \g_a$,  the bracket on $\g$ is $A$-graded, i.e. $[\g_a, \g_b]\subset \g_{a+b}$, and satisfies a generalized  Jacobi identity in which the sign $(-1)^{|x||y|}$ is replaced with $(-1)^{\delta(a)\delta(b)}$, where $\delta:A \to \mathbb{Z}_2$ is a group homomorphism.

\begin{definition}
    Let $\g$ be an $A$-graded Lie superalgebra. Then
    $$
    supp(\g)=\{\alpha\in A\,\,|\,\, \g_{\alpha}\ne \{0\} \}.
    $$
\end{definition}

In this paper most of the time we will work with $\mathbb Z^r$-graded Lie superalgebras. Let $\alpha_1,\ldots, \alpha_r$ are formal variables. 

\begin{definition}\label{def weight system}
	A  weight system is a pair $(\Delta, \chi)$, where $\Delta$ is a subset in $\mathbb Z\alpha_1\oplus \cdots \oplus \mathbb Z\alpha_r$ satisfying the following properties
	\begin{enumerate}
		\item $0 \in \Delta$ and $\alpha_i \in \Delta$, where $i=1,\ldots, r$;
		\item if $\delta\in \Delta$ and $\delta = \sum_i a_i \alpha_i$, where $a_i\in  \mathbb Z$, then $a_i\geqslant 0$,
	\end{enumerate}
	and $\chi:A\to \Z_{2}$, $\delta \mapsto |\delta|$, is a group homomorphism. We call $|\delta|$ the parity of $\delta$.
	A weight system $(\Delta, \chi)$ is called {\it multiplicity free}, if $\delta = \sum a_i \alpha_i\in \Delta$ implies $a_i\in \{0,1\}$.
\end{definition}

Let $(\Delta, \chi)$ be a weight system. Then $\Delta= \Delta_{\bar 0}\cup \Delta_{\bar 1}$, where $\Delta_{\bar 0}$ contains all even weights, while $\Delta_{\bar 1}$ contains all odd ones. Note that $0$ is always in $\Delta_{\bar 0}$.  A Lie super\-algebra $\mathfrak g= \mathfrak g_{\bar 0} \oplus \mathfrak g_{\bar 1}$ is called a graded Lie superalgebra of type $(\Delta, \chi)$, if $\g$ is a $\mathbb Z^r$-graded Lie superalgebra  of type $\Delta$ and
	$$
	\mathfrak g_{\bar 0} =  \bigoplus\limits_{\delta\in \Delta_{\bar 0}} \mathfrak g_{\delta} , \quad \mathfrak g_{\bar 1} =  \bigoplus\limits_{\delta\in \Delta_{\bar 1}} \mathfrak g_{\delta}.
	$$
Very often we will omit $\chi$ and write a Lie superalgebra of type $\Delta$ meaning a Lie superalgebra of type $(\Delta, \chi)$.

\subsection{Supermanifolds, graded manifolds and $n$-fold vector bundles}

\subsubsection{Supermanifolds} 

We consider a complex-analytic supermanifold in the sense of Berezin and Leites \cite{Bern,Leites}, see also \cite{BLMS}. Thus, a supermanifold
$\mcM = (\mcM_0,\mathcal{O}_{\mcM})$ of dimension $n|m$ is a $\mathbb{Z}_2$-graded
ringed space that is locally isomorphic to a super\-domain in
$\mathbb{C}^{n|m}$. Here the underlying space $\mcM_0$ is a complex-analytic manifold.  The dimension $n$ of the underlying manifold $\mcM_0$ is called {\it even dimension of $\mcM$}, while $m$ is called {\it odd dimension of $\mcM$}. A 
morphism $\mcM \to \mathcal{N}$ between two
supermanifolds is a morphism between $\mathbb{Z}_2$-graded
ringed spaces, this is, a pair $F = (F_{0},F^*)$, where $F_{0}:\mcM_0\to \mathcal N_0$ is a holomorphic mapping and $F^*: \mathcal{O}_{\mathcal N}\to (F_{0})_*(\mathcal{O}_{\mathcal M})$ is a homomorphism of sheaves of
$\mathbb{Z}_2$-graded ringed spaces. A morphism $F$ is called an isomorphism if $F$ is inver\-tible. A supermanifold $\mcM$  is called split, if its structure sheaf is isomorphic to $\bigwedge\mathcal E^*$, where $\mathcal E$ is a sheaf of sections of a certain vector bundle $\mathbb E[\bar 1]$. (Here $\mathbb E[\bar 1]$ means a usual vector bundle $\mathbb E$ additionally assumed that its local sections are odd.) We see that in this case the structure sheaf is $\Z$-graded.  A morphism of split supermanifolds is called split if it preserves the fixed $\Z$-gradings. Split supermanifolds with split morphisms form a category of split supermanifolds.

According to the Batchelor--Gawedzki Theorem any smooth supermanifold  is split. This is not true in the complex-analytic case, see \cite{Ber,Green}. However we can define a functor from the category of supermanifols to the category of split supermanifolds.    Let $\mathcal M=(\mathcal M_0,\mathcal O)$ be a supermanifold. (Sometimes we will omit $\mcM$ in $\mathcal O_{\mcM}$ if the meaning of $\mathcal O$ is clear from the context.) Then its structure sheaf possesses the following filtration
\begin{equation}\label{eq filtration}
    \mathcal O = \mathcal J^0 \supset \mathcal J \supset \mathcal J^2 \supset\cdots \supset \mathcal J^p \supset\cdots,
\end{equation}
where $\mathcal J$ is the sheaf of ideals generated by odd elements in $\mathcal O$. We define
$$
\mathrm{gr} (\mathcal M): = (\mathcal M_0,\mathrm{gr}\mathcal O),
$$
where
$$
\mathrm{gr}\mathcal O = \bigoplus_{p \geq 0} \mathcal J^p/\mathcal J^{p+1}.
$$
The supermanifold $\mathrm{gr}(\mathcal M)$ is split,  that is its structure sheaf is isomorphic to $\bigwedge \mathcal E^*$, where $\mathcal E^*= \mathcal J/\mathcal J^{2}$ is a locally free sheaf. Since any morphism $F$ of supermanifolds preserves the filtration (\ref{eq filtration}), the morphism $\gr (F)$ is defined.  Summing up, the functor $\mathrm{gr}$ is a functor from the category of super\-manifolds to the category of split supermanifolds, see for example \cite[Section 3.1]{Vish_funk} for details. We can apply the functor $\gr$ to a Lie supergroup $\mcG$ and we will get a split Lie supergroup $\gr(\mcG)$.  Later we also will define a corresponding functor $\gr'$ for Lie superalgebras so that $\Lie \circ \gr (\mcG) = \gr' \circ \Lie (\mcG)$ for any Lie supergroup $\mcG$.

\subsubsection{Graded manifolds of type $\Delta$, graded manifolds of degree $n$ and $r$-fold vector bundles}

Let us start with a notion of a graded manifold of type $(\Delta,\chi)$, where $(\Delta,\chi)$ is as in Definition \ref{def weight system}. Again very often we will omit $\chi$, and write a graded manifold of type $\Delta$.   Let
 $$
  V= \bigoplus_{\delta \in \Delta} V_{\delta}, \quad V_{\bar 0}= \bigoplus_{\delta \in \Delta_{\bar 0}} V_{\delta}, \quad V_{\bar 1}= \bigoplus_{\delta \in \Delta_{\bar 1}} V_{\delta}.
 $$
 Consider a $\mathbb Z^r$-graded ringed space $\mathcal U = (\mathcal U_0,\mathcal O_{\mathcal U})$, where $\mathcal U_0 \subset V^*_0$ is an open set, 
$\mathcal O_{\mathcal U}: = \mathcal F_{\mathcal U_0}\otimes_{\mathbb K} S^*(V)$  and
 $\mathcal F_{\mathcal U_0}$ is the sheaf of smooth or holomorphic  functions on $\mathcal U_0$. 
 
 \begin{definition}\label{def graded domain}
     We call the ringed space $\mathcal U$ a  graded domain of type $(\Delta,\chi)$ and of dimension $\{n_{\delta}\}_{\delta\in \Delta}$, where $n_{\delta} = \dim V_{\delta} $.
 \end{definition}
 
 Let us choose a basis  $(x_i)$ in $V_0$ and a basis $(\xi_j^{\delta})$ in any $ V_{\delta}$. Then the set $(x_i, \xi_j^{\delta})_{\delta \in \Delta\setminus \{0\}}$ is a system of local coordinates on $\mathcal U$. We assign the weight $0$ and the parity $\bar 0$ to $x_i$ and the weight $\delta$ and the parity $|\delta|$ to $\xi_j^{\delta}$. Such coordinates $(x_i, \xi_j^{\delta})$ are called graded.

\begin{definition}\label{def graded man of tupe Delta}
    A graded manifold of type $(\Delta,\chi)$ and of dimension $\{n_{\delta}\}_{\delta \in \Delta}$ is a $\mathbb Z^r$-graded ringed space $\mathcal N = (\mathcal N_0, \mathcal O_{\mathcal N})$, that is locally isomorphic to a graded domain of type $(\Delta,\chi)$ and of dimension $\{n_{\delta}\}_{\delta \in \Delta}.$

 A  morphism of graded manifolds of type $(\Delta,\chi)$ is a morphism of the corresponding $\mathbb Z^r$-graded ringed spaces. 
\end{definition}  

If $\Delta = \{0,1,\ldots, n\}$ and $\chi(1)=\bar 1$, a graded manifold of type $\Delta$ is also called a {\it graded manifold of degree $n$}. (See \cite{Roytenberg} for more information about graded manifolds of degree $n$.)  If $\Delta$ is multiplicity free, this is if $\delta = \sum_i a_i \alpha_i \in \Delta$ then $a_i\in \{0,1\}$, a graded manifold of type $\Delta$ is called an {\it $r$-fold vector bundle of type $\Delta$}.  This definition of an $r$-fold vector bundle is equivalent to a classical one as it was shown in \cite[Theorem 4.1]{GR}.  

\begin{remark}\label{rem projection of graded}
Now let $\Delta$ be as in Definition \ref{def weight system} and $\Delta'\subset \Delta$ satisfies the following property: if $\delta\in \Delta'$ and $\delta= \sum_i \delta_i$ for $\delta_i\in \Delta$, then $\delta_i$ are also in $\Delta'$. In this case to any graded manifold $\mathcal N$ of type $\Delta$ we can assign a graded manifold $\mathcal N'$ of type $\Delta'$.   A detailed description of this well-known construction can be found for example in \cite[Section 4.1]{Vish}. For instance to any graded manifold of degree $n$ we can assign a graded manifold of degree $0\leq n'<n$. In this case we have a natural morphism $\mathcal N\to \mathcal N'$, which is called projection. 
\end{remark}

\subsection{Lie supergroups and super Harish-Chandra pairs}
\subsubsection{Lie supergroups}
A {\it Lie supergroup} is a group object in
the category of supermanifolds.
In other words a Lie supergroup is a supermanifold $\mathcal G=(\mathcal G_0, \mathcal O_{\mathcal G})$ with three morphisms:
$\mu:\mathcal G\times \mathcal G\to \mathcal G$
(the multiplication morphism), $\kappa:\mathcal G\to \mathcal G$ (the inversion morphism),
$\varepsilon:(\mathrm{pt},\mathbb C)\to \mathcal G$ (the identity
morphism). Moreover, these morphisms have to satisfy the usual
conditions, modeling the group axioms. The underlying manifold $\mathcal G_0$ of $\mathcal G$
is a smooth or complex-analytic Lie group. As in the theory of Lie groups and Lie algebras, we can define the Lie superalgebra $\mathfrak{g}=\Lie(\mcG)$ of a Lie group $\mathcal G$. By definition, $\mathfrak{g}$ is the subalgebra of the Lie superalgebra of vector fields on
$\mathcal G$ consisting of all right invariant
vector fields on $\mathcal G$. Any
right invariant vector field $Y\in \mathfrak{g}$ has the following form
\begin{equation}
\label{left inv vect field}
Y=(Y_e\otimes \mathrm{id})\circ \mu^*,
\end{equation}
where $Y_e\in (\mathfrak m_e/\mathfrak m^2_e)^*$ and $\mathfrak m_e$ is the maximal ideal of $(\mathcal O_{\mathcal G})_e$. Note that as the case of classical geometry the vector superspace $(\mathfrak m_e/\mathfrak m^2_e)^*$ may be identified with the vector superspace of all maps $D: (\mathcal O_{\mathcal G})_e \to \mathbb{C}$ satisfying the Leibniz rule. 
Further the map $Y_e\mapsto Y= (Y_e\otimes \mathrm{id})\circ \mu^*$ is an
isomorphism of $(\mathfrak m_e/\mathfrak m^2_e)^*$ onto
$\mathfrak{g}$.

Similarly we can define a {\it graded Lie supergroup of type $\Delta$} as a group object in the category of graded manifolds of type $\Delta$ or a {\it graded Lie supergroup of degree $n$} as a group object in the category of graded manifolds of degree $n$. 
 Let us consider some examples of Lie supergroups with an additional gradation in the structure sheaf. 

\begin{example}[Lie supergroup $GL(V)$] The space of endomorphisms of a vector superspace $V = V_{\bar{0}} \oplus V_{\bar{1}}$ has a natural gradation by integer numbers $-1, 0, 1$:
$$
\End(V) = \bigoplus\limits_{i=-1, 0, 1} \End^i(V),
$$ 
where 
\begin{align*}
    \End(V)^0 = \End(V_{\bar{0}}) &\oplus \End(V_{\bar{1}}),\quad \End(V)^1 = \Hom(V_{\bar{0}},V_{\bar{1}}),\\
    &\End(V)^{-1} = \Hom(V_{\bar{1}},V_{\bar{0}})
\end{align*}
giving rise to a $\Z$-graded Lie superalgebra $\mathfrak{gl}(V)$. It integrates to the  $\Z$-graded Lie supergroup $\mathrm{GL}(V)$. Consider the case $\dim V =(1|1)$. The entries of a matrix 
$$
A = \begin{bmatrix}
       a & \alpha \\ \beta & b
    \end{bmatrix}\in \mathrm{GL}(\mathbb K^{1|1}),
    $$ 
    constitute a graded coordinate domain, the subdomain of $\mathbb K^{2|2}$, the coordinates $a$, $b$ are of degree $0$, and $\alpha$, $\beta$ are of degrees $-1$ and $1$, respectively. The matrix multiplication respects the assigned grading, e.g. $m^* (\alpha) = a' \alpha'' + \alpha' b''$ is of degree $0=0+0= 1+(-1)$. Note that $\mathrm{GL}(V)$ with this gradation is not a graded manifold of degree $n$, since this supermanifold possesses graded coordinates of negative degree.  
\end{example}

\begin{remark}
In this paper we denote by $\oT$ the {\bf antitangent functor}. This means that  $\oT \mcM$ is the vector bundle obtained from the tangent bundle $\T \mcM$ by means of the parity functor. If $(x^A)$ are even/odd coordinates on $\mcM$ then the fiber coordinates on $\oT \mcM$  will be denoted by $(\dd x^A)$ and the parity of $\dd x^A$ is the parity of $x^A$ plus one modulo two. 
\end{remark}

A canonical example of graded Lie supergroups of degree $n$ comes from the tangent lifts. 
\begin{example} Higher order tangent bundles, not to be confused with the related iterated
tangent bundles, are the natural geometric home of higher derivative Lagrangian mechanics \cite{Bruce}.
A higher order ($k$-order)  tangent bundle $\T^{k} \mathcal{M}$ of a supermanifold $\mathcal{M}$ can be defined in a natural way. By applying the functor $\T^k$ to the Lie supergroup $\mcG$ and the structure morphisms of $\mcG$  we get a graded Lie supergroup of type $(\Delta,\chi)$, where  $\Delta=\{0, 1, \ldots, k\}$ and {\newMR $\chi(1)=\bar 0$}.  
\end{example}
\commentMR{There is no superization of second order graded bundles. This means there is no passage from $\T^2 \mcM$ to $\oT^2 \mcM$ as there is no equivalence between symmetric and skew-symmetric DVB. $\T^2 \mcM$ has jet of curves space interpretation. $\oT^2 \mcM$ is only a generalized suepermanifold --- better not to include this in present paper.}
\commentMR{This example was completely wrong - needs clarification and correction. We have to things: iterated tangent or anti-tangent bundle and $k$-order tangent bundle. However, $k$-order tangent bundle is not usually a supermanifold but a generalized supermanifold! It is given by functor of points - but this functor is nor representable. }
\commentMR{I suggest to add A. Bruce, ON CURVES AND JETS OF CURVES ON SUPERMANIFOLDS. This will give us some citations. Note:
Higher order tangent bundles, not to be confused with the related iterated
tangent bundles, are the natural geometric  and thus a clear geometric understanding of the super-versions is important
for general supermechanics. For time-being we can omit this example.}


Another interesting example of $\Z$-gradation on a Lie supergroup was noticed by V.~Kac: 
\begin{example} Let $V = V_{\bar{0}} \oplus V_{\bar{1}}$ be a vector superspace equipped with a non-degenerate even symmetric bilinear form $Q$, i.e. 
$Q(x, y) = (-1)^{|x||y|} Q(y, x)$ for homogeneous $x, y\in V$ and $Q:V\times V\to \mathbb K$ is an even map.   It follows that $V_{\bar{0}} \perp V_{\bar{1}}$ and $Q$ restricted to $V_{\bar{0}}$ (resp. $V_{\bar{1}}$) is symmetric (resp. skew-symmetric). Thus the dimension of $V_{\bar{1}}$ is even and one can find   lagrangian subspaces $L, L' \subset V_{\bar{1}}$ such that $V_{\bar{1}} = L \oplus L'$, where $L'$ is naturally identified with the dual to $L$ and $Q|_{V_{\bar{1}}}$ has the following natural form with respect to this decomposition:  
$$
Q((l_1, l_1'), (l_2, l_2')) = l_1'(l_2) - l_2'(l_1),
$$     
where $l_i\in L$ and $l'_i\in L'$. 
The orthosymplectic Lie superalgebra $\mathfrak{osp}(V, Q)$
is spanned by homogeneous endomorphism $T: V\to V$ such that 
$$
    Q(T x, y) + (-1)^{|T||x|} Q(x, Ty)=0
$$
for any homogeneous $x, y\in V$. 
It follows that $\mathfrak{osp}(V, Q)_{\bar{0}} = \mathfrak{so}(V_{\bar 0})\oplus \mathfrak{sp}(V_{\bar{1}})$ and $\mathfrak{osp}(V, Q)_{\bar{1}} = V_{\bar 0} \otimes V_{\bar 1}$. There are well known isomorphisms $\mathfrak{so}(W) = W \wedge W$ and $\mathfrak{sp}(W) = S^2 W$. Assigning the weight $1$ to $L$ and the weight $-1$ to $L'$ one find  a canonical $\Z$-grading on $\mathfrak{osp}(V)$ with non-zero parts in degrees $\pm2$, $\pm1$ and $0$:  
$$
\mathfrak{osp}(V) = S^2(L') \oplus (V_{\bar 0} \otimes L') \oplus (L' \otimes L \oplus \V_{\bar 0}\wedge V_{\bar 0}) \oplus (V_{\bar 0} \otimes L) \oplus S^2 L.  
$$
Thus the corresponding Lie supergroup, the orthosymplectic Lie supergroup, is also naturally $\mathbb{Z}$-graded.
\end{example}

\subsubsection{Super Harish-Chandra pairs}
To study a Lie supergroup one uses the theo\-ry of super Harish-Chandra pairs, see \cite{Bern} and also \cite{BCC,Fioresi,ViLieSupergroup}.

\begin{definition}
	A  super Harish-Chandra pair $(G, \g)$  consists of 
	a Lie superalgebra $\g=\mathfrak{g}_{\bar
		0}\oplus\mathfrak{g}_{\bar 1}$,    a Lie group $G$    such that $\mathfrak{g}_{\bar 0} = \mathrm{Lie} (G)$, and
	  an action $\alpha$  of $G$ on $\g$ by authomorphisms such that 
	  \begin{itemize}
	    \item for any $g\in G$ the action $\alpha(g)$ restricted to $\g_{\bar{0}}$ coincides with the adjoint action $Ad(g)$;
	\item the differential $(d\alpha)_e: \g_0 \to \T_e\Aut(\g) =  \Der(\g)$ at the identity $e\in G$ coincides with
		the adjoint representation $\ad$ of $\mathfrak g_{\bar 0}$ in $\mathfrak g$.
	\end{itemize}
\end{definition}
Super Harish-Chandra pairs form a category.  
 The following theorem was proved in \cite{ViLieSupergroup} for complex-analytic Lie supergroups.

\begin{theorem}\label{theor HC pairs}
	The category of complex-analytic Lie supergroups is
	equivalent to the category of complex-analytic super Harish-Chandra pairs.
\end{theorem}

A similar result in the real case was obtained in \cite{Kostant}, the algebraic case was treated by several authors, see for example \cite{Gav,Masuoka,MasuokaShibata} and references therein.  Theorem \ref{theor HC pairs} implies that any Lie supergroup is globally split, that is, its structure sheaf is isomorphic to a wedge product of a certain trivial vector bundle. More explanation of this fact will be given later.

 The notion of super Harish-Chandra pair has an obvious $A$-graded generalization, where $A$ is an abelian group. Namely, we assume that $\g$ is an $A$-graded Lie superalgebra, the Lie algebra of $G$ is $\g_0$, where $\g_0$ is the part of $\g$ in degree 0, $\alpha(g): \g \to \g$, $g\in G$, is an automorphism of $\g$, in particular it preserves the $A$-gradation and  $\dd_e\alpha (X) = [X, \cdot]_\g$ for any $X\in \g_0$. In \cite[Theorem 5.6]{Jubin} it was noticed that an analogue of Theorem \ref{theor HC pairs} holds true for graded Lie supergroups of degree $n$.
We will also work with super Harish-Chandra pairs of type $\Delta$. A super Harish-Chandra pairs of type $\Delta$ is a $\mathbb Z^r$-graded super Harish-Chandra pair $(G,\g)$, where $\g$ is a  $\mathbb Z^r$-graded Lie superalgebra of type $\Delta$. Recall that in this case $\mathbb Z^r$ is generated by the set $\Delta$, with additional agreement about parities of generators of $\Delta$, see Definition \ref{def weight system}. In this case repeating argument \cite[Theorem 5.6]{Jubin}, we get an analogue of Theorem \ref{theor HC pairs} as well.

\begin{example} Let $V = \oplus_{k=0}^n V_k$ be a non-negatively $\Z$-graded vector space with $\dim V < \infty$. Denote by $\End_q(V)$, where $q\in Z $, the space of endomorphisms $T: V\to V$ increasing the degree by $q$, i.e. $T: V\to V$  such that $T|_{V_i}: V_i \to V_{i+q}$. Thus $\und{\End}(V) =  \oplus_{q=-n}^n \End_q(V)$ is  a $\Z$-graded vector space. It will be consi\-dered as a  $\Z$-graded Lie superalgebra with the bracket given by the graded commutator.  The corresponding $\Z$-graded Lie supergroup $\mathrm{GL}(V)$ has the body $G_0 = \times_{i=0}^n GL(V_i)$ and its sheaf of functions is given by the super symmetric algebra over $F = \oplus_{0\leq i\neq j\leq n} V_i \otimes V_j^*$,
as $V_i \otimes V_j^* = \Hom(V_i, V_j)^*$, 
i.e. $\mcO(U) = \mathcal{C}^\infty(U)\otimes \Sym^{} F$. The gradation in $F$ is obvious, $|V_i\otimes V_j^*| = i-j$ and the group multiplication preserve this gradation. 
The graded Harish-Chandra pair corresponding to $GL(V)$ is $(G_0, \und{\End}(V))$ equipped with the adjoint action of $G_0$ on 
$\und{\End}(V)$.  The language of Harish-Chandra pairs allows us to study some infinite dimensional generalizations of Lie supergroups in a convienient way. We simply drop the assumption that $\g$ has  finite dimension. For example, if $V = \oplus_{k=0}^{\infty} V_k$ is a $\Z$-graded vector space such that $\dim V_k < \infty$, but the dimension of $V$ can be infinite, then $GL(V)$ can be understood as a pair consisting of the direct sum  of $GL(V_i)$,  $G_0 = \oplus_{i=0}^{\infty} GL(V_i)$ and the natural action of  $G_0$ on $\und{\End}(V) = \oplus_{i, j\geq 0} \Hom(V_i, V_j)$. 
\end{example}



\subsubsection{Construction of the Lie supergroup corresponding to a super Harish-Chandra pair}
Let us remind a reader how to construct a Lie supergroup $\mathcal G$ (or a graded Lie supergroup of type $\Delta$) using a given super (or of type $\Delta$) Harish-Chandra pair $(G,\mathfrak g)$. Denote by $\mathcal U(\mathfrak h)$ the universal enveloping algebra of Lie superalgebra $\mathfrak h$.  We need to define a structure sheaf $\mathcal O$ of $\mathcal G$. In the super and graded cases we, respectively, put
\begin{equation}\label{eq structure sheaf of G}
\begin{split}
\mathcal O = \mathrm{Hom}_{\mathcal U(\mathfrak g_{\bar 0})} (\mathcal U(\mathfrak g), \mathcal F_{G}), \quad \mathcal O = \mathrm{Hom}_{\mathcal U(\mathfrak g_{0})} (\mathcal U(\mathfrak g), \mathcal F_{G}).
\end{split}
\end{equation}
Here $\mathcal F_G$ is the sheaf of (holomorphic) functions on $G$. 
Using the Hopf algebra structure on $\mathcal U(\mathfrak g)$ we can define explicitly the multiplication morphism $\mu$, the inversion morphism $\kappa$ and the identity $\varepsilon$, see for instance \cite{ViLieSupergroup}. Indeed, assume that a super  or graded of type $\Delta$ Harish-Chandra pair $(G,\mathfrak g)$ is given. Let us define the supergroup structure of the corresponding Lie supergroup or graded Lie supergroup  $\mathcal G$. Let $X\cdot Y\in
\mathcal U(\mathfrak{g}\oplus \mathfrak{g})\simeq
\mathcal U(\mathfrak{g})\otimes \mathcal U(\mathfrak{g})$, where
$X$ is from the first copy of $\mathcal U(\mathfrak{g})$ and $Y$
from the second one, $f\in \mathcal{O}$, see (\ref{eq structure sheaf of G}), and $g,\,h\in
G$. The following formulas define a multiplication morphism, an
inverse morphism and an identity morphism respectively:
\begin{equation}\label{eq umnozh in supergroup}
\begin{split}
\mu^*(f)(X\cdot Y)(g,h)&=f(\alpha(h^{-1})(X)\cdot Y)(gh);\\
\kappa^*(f)(X)(g^{-1})&=f(\alpha(g^{-1})(S(X)))(g);\\
\varepsilon^*(f)&=f(1)(e).
\end{split}
\end{equation}
Here $S$ is the antipode map in $\mathcal U(\mathfrak{g})$ considered as a Hopf algebra and $\alpha$ is as in the definition of a super Harish-Chandra pair. 

\subsection{(Skew-symmetric) double vector bundles} \label{sec_DVBs}

A {\it double vector bundle} (\DVB, in short) is a graded manifold of type $\Delta=\{0, \alpha, \beta, \alpha+\beta\}$. In this subsection we assume that the weights $\alpha$ and $\beta$  are even, however later on we shall drop this assumption.  
\commentMR{Here we can add coordinate description, but this is obvious.}
Geometrically we can see a double vector bundle as a quadruple 
$(D; A,  B; M)$  
with  the following diagram of morphisms
\begin{equation}\label{eq DVB}
    \begin{tikzcd}
D \arrow[r, ] \arrow[d]
&  A \arrow[d,  ] \\
 B \arrow[r,  ]
& M
\end{tikzcd},
\end{equation}
where all maps are bundle projections and there are imposed some natural compatibility conditions, see \cite{Pradines:1974, Mackenzie}. \commentMR{We should also write some words on history: Pradines}

{\newMR In particular, $D$ has two  vector bundle structures,  one over the manifold $A$ and the second over $B$. Moreover, $A$, $B$ are vector bundles over $M$ called \emph{the side bundles} of $D$.  The compatibility condition can be easily  expressed using Euler vector fields: $[\Euler_A, \Euler_B] = 0$ \cite{GR}. Essentially, a DVB can be consider as a manifold equipped with two Euler vector fields, $D = (D; \Euler_A, \Euler_B)$, describing the vector bundle structure on \emph{the legs} $D\to A$, $D\to B$ of $D$. This lead to many simplifications in the theory of \DVBs.}  

{\newMR The core $C$ of \DVB  is defined as intersection  of the kernels of two bundle projections $D  \to A$ and $ D  \to  B$. It is a vector bundle over $M$ with the Euler vector field defined by the restriction, $\Euler_A|_C = \Euler_B|_C$.}  

{\newMR The \emph{flip} of a DVB is obtained by interchanging the legs of $D$, ie. the flip of $D =(D; \Euler_A, \Euler_B)$ is $(D; \Euler_B, \Euler_A)$.}


Symmetric and skew-symmetric DVBs are DVBs $(D; A, B; M)$  with an involution $\sigma: D\to D$ satisfying a certain addition condition. They were introduced in
\cite{Grab} in order to recognize the image of \emph{the linearization functor} which associates with a graded manifold of type $\Delta = \{0, \alpha, 2 \alpha\}$  a DVB. Under this association graded manifolds of type $\Delta$ with even $\alpha$ (i.e., purely even graded bundles of degree 2) are in one-to-one correspondence with  symetric DVBs, while   graded manifolds of type $\Delta$ with odd $\alpha$ (i.e., $N$-manifolds of degree 2)  ---  with skew-symmetric DVBs. The results are extended to any order. For the purposes of this manuscript it is enough to recall a definition of skew-symmetric \DVBs.

\begin{definition} A skew-symmetric $\DVB$ is a pair $(D, \sigma)$ consisting of a \DVB $D$
and an involution $\sigma: D\to D$  (i.e. $\sigma \circ \sigma = \id$) such that 
\begin{enumerate}[(i)]
    \item  $\sigma$ exchanges the legs of $D$, i.e. $\sigma$ is a DVB morphism from $D$ to the flip of $D$, 
    \item the restriction of $\sigma$ to the core  is minus the identity. 
\end{enumerate}
A morphism between skew-symmetric  \DVB s is assumed to intertwines with their involutions.
\end{definition}
It follows immediately that the side bundles of a skew-symmetric DVB are isomorphic and that  any skew-symmetric DVB admits an atlas with graded coordinates $(x^a; y^i_\alpha, Y^j_\beta; z^\mu_{\alpha+\beta})$ such that
\begin{itemize}
    \item the transition functions for $y^i$ and $Y^i$ are the same and 
    $$
    z^\mu_{ij} = Q^\mu_\nu z^\nu + \frac12 Q^\mu_{ij} y^i Y^j
    $$
    with $Q^\mu_{ij} =  - Q^\mu_{ji}$, 
\item the involution $\sigma$ has a form: $\sigma^*(y^i) = Y^i$ (hence $\sigma^*(Y^i) = y^i$) and $\sigma^*(z^\mu) = -z^\mu + \sigma^\mu_{ij}(x) y^i Y^j$ with $\sigma^\mu_{ij} =   \sigma^\mu_{ji}$.   
\end{itemize}


To any double vector bundle $D$ (with the core $C$ and side bundles $A$, $B$) we can assign  a short exact sequence 
\begin{equation}\label{eq exact for DBV}
     0 \to C \to \wh{D} \to  A\otimes  B \to 0,
\end{equation}
which is obtained by dualizing the short exact sequence  
$$
0 \to \ker \pr_C \to \mcA^{\alpha+\beta}(D) \xrightarrow{\pr_C}  C^* \to 0
$$
where $\mcA^{\alpha+\beta}(D)$ is the $\mathcal F (M)$-module of homogeneaus functions on $D$ of weight $\alpha + \beta$ and $\pr_C$ denotes the restriction of a function $f$ to the core $C$, ie. $\pr_C(f) = f|_C$.  Locally, $\mcA^{\alpha+\beta}(D)$ is generated by the functions 
{\newMR $y_{\alpha}^i$, $Y^j_{\beta}$ and $z^\mu_{\alpha+\beta}$. In other words, } the dual to  $\wh{D}$ is the vector bundle whose space of sections is $\mcA^{\alpha+\beta}(D)$.

If $(D, \sigma)$ is skew-symmetric then we can decompose 
$$\mcA^{\alpha+\beta}(D) = \mcA^{\alpha+\beta}_{+}(D) \oplus  \mcA^{\alpha+\beta}_{-}(D), $$
where $\mcA^{\alpha+\beta}_{\vareps}(D)$ consists of those $f\in \mcA^{\alpha+\beta}(D)$ such that $\sigma^*(f) = \vareps f$ and $\vareps$ is plus  or minus.    

\begin{theorem} Let $(D, \sigma)$ be a skew-symmetric $\DVB$ with side bundle $E\to M$. Then  $\mcA^{\alpha+\beta}_{+}(D)$ coincides with the space of sections of $S^2 E^*$ while $\mcA^{\alpha+\beta}_{-}(D)$  gives rise to a short exact sequence 
\begin{equation}\label{e:ses-skewDVB}
    0\to \Gamma(\bigwedge^2 E^*) \to \mcA^{\alpha+\beta}_{-}(D) \to \Gamma(C^*)\to 0.
\end{equation}
The skew-symmetric $\DVB$ $(D, \sigma)$ can be reconstructed completely from the above sequence. 
\end{theorem}
\begin{proof}
Let $(x^a, y^i, Y^j, z^\mu)$ be adopted coordinates for a skew-symmetric DVB $(D, \sigma)$. In these coordinates the decomposition of $\mcA^{\alpha+\beta}(D)$ is clear: $\mcA^{\alpha+\beta}_{+}(D)$ is locally generated by functions $y^i Y^j + y^j Y^i$ and $\mcA^{\alpha+\beta}_{-}(D)$ by $z^\mu$ and $y^i Y^j - y^j Y^i$. The projection on $\Gamma(C^*)$ is the same, it is given as the restriction to the core bundle of $D$. 
 \end{proof}

\subsection{Inverse limit of supermanifolds and Lie supergroups}\label{sec inverse limit} 

We shall work with some types of infinite dimensional supermanifolds and Lie supergroups. There is no need to present a general theory as all  examples we shall work with have a form of an inverse limit of supermanifolds (even with the same body $\mathcal M_0$): 
$$
    \mcM_1 \xleftarrow{\pr^2_1} \mcM_2 \xleftarrow{\pr^3_2} \mcM_3 \xleftarrow{} \ldots
$$
where $\mcM_i = (\mathcal M_0, \mathcal{O}_{\mcM_i})$ and $\pr^{k+1}_k:\mcM_{k+1}\rightarrow\mcM_k$ is the projection of the graded manifold $\mcM_{k+1}$ of degree $k+1$ to the graded manifold $\mcM_{k}$ of degree $k$, see Remark \ref{rem projection of graded}. We define
$$
\mcM_{\infty} = (\mathcal M_0, \mathcal O_{\mcM_{\infty}}), \quad \text{ where } \mcO_{\mcM_{\infty}}  = \bigcup_{k=1}^\infty \mcO_{\mcM_{k}}, 
$$
i.e. $f\in \mcO_{\mcM_{\infty}}$ if and only if there exist $k$ such that $f\in \mcO_{\mcM_k}$. A morphism  $f: \mcM_{\infty} \to \mathcal N_{\infty}$ is a family of morphisms $f = (f_k: \mcM_k \to \mathcal N_{a_k})$ 
where $(a_k)$ is a non-decreasing sequence of positive integers and the family $(f_k)$ is compatible with the projections: 
$$
\xymatrix{
\mcM_{k+1} \ar[rr]^{f_{k+1}} \ar[d]^{\pr^{k+1}_k} && \mathcal N_{a_{k+1}} \ar[d]^{\pr^{a_{k+1}}_{a_k}} \\
\mcM_k \ar[rr]^{f_{k}} && \mathcal N_{a_k}
}
$$
where $\pr^{a_{k+1}}_{a_k}$ is the composition of projections $\pr^{i+1}_i$ for $i=a_{k+1}-1$ to $a_k$.

The inverse limit of Lie supergroups has a Lie supergroup structure. We additionally have to assume that the projections $\pr^{i+1}_i$ are Lie supergroup homomorphisms.  Also the inverse limit of Lie (super)algebras has a Lie (super)algebra structure thus a Lie functor makes sense for inverse limit of Lie supergroups. For example a $\mathbb{Z}$-graded Lie superalgebra $\g = \bigoplus_{k=0}^{\infty} \g_k$ is an inverse limit of Lie algebras $\g/I_k$ where $I_k = \g_{k+1}\oplus \g_{k+2}\oplus \ldots$. There is one-to-one correspondence of the inverse limit of graded Lie supergroups and the inverse limit of their graded Harish-Chandra pairs.

\section{The Donagi--Witten construction, supermanifolds, double vector bundles and graded manifolds of degree $2$}\label{sec Witten}

In \cite[Section 2.1]{Witten Atiyah classes} Donagi and Witten gave a description of the first obstruction class $\omega=\omega_2$ via differential operators. In this section we remind the definition of $\omega_2$ and the Donagi--Witten construction. Further we give an interpretation of the Donagi--Witten construction using the language of  double vector bundles and graded manifolds of degree $2$. At the end using differential forms instead of differential operators we simplify this construction, which allows us to find its higher analogue. 

\subsection{First obstruction class $\omega_2$} 

 Let us describe the first obstruction class to splitting a supermanifold using results \cite{Ber,Green,Oni,Roth}. We follow the exposition of \cite{Oni}. First of all consider a split supermanifold  $\mcM=(\mcM_0,\mcO)$, where $\mcO=\bigwedge \mathcal E^*$ and $\mathcal E$ is the sheaf of sections of $\mathbb E[\bar 1]$. Denote by $\mathcal{D}er\mathcal O$ the sheaf of vector fields on $\mcM$ and by $\mathcal{D}er\mathcal F$ the sheaf of vector fields on the underlying space $\mcM_0$.  The sheaf $\mathcal{D}er\mathcal O=\bigoplus_{p\geq -1} \mathcal{D}er_p\mathcal O$ is naturally $\Z$-graded since $\bigwedge \mathcal E^*$ is $\Z$-graded.  We have the following exact sequence
\begin{equation}\label{eq exact seq}
    0\to \bigwedge^3\mathcal E^*\otimes\mathcal E \to \mathcal{D}er_2\mathcal O \to \bigwedge^2\mathcal E^*\otimes\mathcal{D}er\mathcal F \to 0,
\end{equation}
see \cite[Formula (5)]{Oni}.

According Green \cite{Green} we can describe all non-split supermanifolds $\mathcal N$ such that $\gr (\mathcal N)\simeq \mcM$ using the sheaf of automorphisms $\mathcal{A}ut \mathcal O$  of  $\mathcal O$. More precisely  consider the following subsheaf of $\mathcal{A}ut \mathcal O$:
$$
\mathcal{A}ut_{(2)} \mathcal O = \{ a\in \mathcal{A}ut \mathcal O\,\,|\,\, a(u)-u\in \mathcal J^2\,\, \text{for any}\,\, u\in \mathcal O \},
$$
see also \cite[Formula (17)]{Oni}. Recall that $\mathcal J$ is the sheaf of ideals generated by odd elements in $\mathcal O$.  Denote by $\mathrm{Aut} (\mathbb E^*)$ the group of automorthisms of $\mathbb E^*$.  There is  a natural action of $\mathrm{Aut} (\mathbb E^*)$ on the sheaf $\mathcal{A}ut_{(2)} \mathcal O$, see \cite[Section 1.4]{Oni}. This action induces an action of $\mathrm{Aut} (\mathbb E^*)$ on the set $H^1(\mcM_0, \mathcal{A}ut_{(2)} \mathcal O)$. 
By Green \cite{Green} points of the set of orbits  $H^1(\mcM_0, \mathcal{A}ut_{(2)} \mathcal O)/ \mathrm{Aut} (\mathbb E^*)$ are in one-to-one correspondence with  isomorphism classes of supermanifolds $\mathcal N$ such that $\gr (\mathcal N)\simeq \mcM$. More precisely to any supermanifold $\mathcal N$ such that $\gr (\mathcal N)\simeq \mcM$ we can assign  a class $\gamma\in H^1(\mcM_0, \mathcal{A}ut_{(2)} \mathcal O)$. If $\gamma_i$ is the class corresponding to a supermanifold $\mathcal N_i$ such that $\gr (\mathcal N_i)\simeq \mcM$, where $i=1,2$. Then $\mathcal N_1\simeq \mathcal N_2 $ if and only if $\gamma_1$ and $\gamma_2$ are in the same orbit of $\mathrm{Aut} (\mathbb E^*)$.

In \cite{Roth} the following map of sheaves was defined
\begin{equation}\label{eq Roth map}
    \mathcal{A}ut_{(2)} \mathcal O \to \mathcal{D}er_2\mathcal O,
\end{equation}
see also \cite[Formula (19)]{Oni}. Combining the map (\ref{eq Roth map}) and the map $\mathcal{D}er_2\mathcal O \to \bigwedge^3\mathcal E^*\otimes\mathcal{D}er\mathcal F$ from (\ref{eq exact seq}) we get the following map of cohomology sets
$$
H^1(\mcM_0, \mathcal{A}ut_{(2)} \mathcal O) \to H^1(\mcM_0, \bigwedge^2\mathcal E^*\otimes\mathcal{D}er\mathcal F)
$$
and the corresponding map of $\mathrm{Aut} (\mathbb E^*)$-orbits 
$$
H^1(\mcM_0, \mathcal{A}ut_{(2)} \mathcal O)/\mathrm{Aut} (\mathbb E^*) \to H^1(\mcM_0, \bigwedge^2\mathcal E^*\otimes\mathcal{D}er\mathcal F)/\mathrm{Aut} (\mathbb E^*).
$$

 If $\gamma\in H^1(\mcM_0, \mathcal{A}ut_{(2)} \mathcal O)$ corresponds to a non-split supermanifold $\mathcal N$, the image of $\gamma$ in $H^1(\mcM_0, \bigwedge^2\mathcal E^*\otimes\mathcal{D}er\mathcal F)$ is called {\it the first obstruction class} to  splitting of $\mathcal N$ and following \cite[Section 2.1]{Witten Atiyah classes} we denote this class by $\omega_2$. 

Consider the case when $\mcM$ has odd dimension $2$ in details. Since $\bigwedge^3 \mathcal E^*=\{0\}$, from (\ref{eq exact seq}) it follows that 
 $$
 \mathcal{D}er_2\mathcal O\simeq \bigwedge^2\mathcal E^*\otimes\mathcal{D}er\mathcal F.  
 $$
In this case the map (\ref{eq Roth map}) is an isomorphism. Therefore we have the following set bijection
$$
H^1(\mcM_0, \mathcal{A}ut_{(2)} \mathcal O) \simeq H^1(\mcM_0, \bigwedge^2\mathcal E^*\otimes\mathcal{D}er\mathcal F)
$$
and the corresponding bijection of the sets of orbits. 
Now Green's result \cite{Green} implies that $\omega_2$ is the only obstruction for a supermanifold to be split in this case. In other words a supermanifold $\mathcal N$ of odd dimension $2$ such that $\gr(\mathcal N)\simeq \mcM$ is split if and only if $\omega_2=0$. Note that the notion of a split and a projectable, see \cite{Witten not projected} for details, supermanifold coincide in this case.

\subsection{The Donagi and Witten construction}
\commentMR{It should be clear what we call the Donagi-Witten construction but this subsection does not explain it. I only guess: the module $\mathcal D^2_{-} |_{\mcM_0}$ and that the obstruction class $\omega$ defined above is the Atiyah class of the extension
$$
0\to \oT M \to D_\omega \to \bigwedge^2 E \to  0  
$$
where sections of $D_\omega$ are identified with $\mathcal D^2_{-} |_{\mcM_0}$. 
Then our observation is that $D_\omega$ gives rise to a skew-symmetric double vector bundle. I am right? 
}
In \cite[Section 2.1]{Witten Atiyah classes} the first obstruction class $\omega_2$   to splitting a supermanifold $\mathcal M=(\mcM_0, \mathcal O)$ was interpreted in terms of a certain sheaf  of differential operators on $\mcM$. Namely, 
the obstruction class $\omega$ defined above is the Atiyah class of the extension
$$
0\to \T M \to D_\omega \to \bigwedge^2 E \to  0  
$$
where sections of $D_\omega$ are identified with some factor of $\mathcal D^2_{-} |_{\mcM_0}$ where the meaning of the latter is explained below.

Let us remind this construction using charts and local coordinates.  Consider two charts $\mathcal U_1$ and $\mathcal U_2$  on $\mcM$ with non-empty intersection and with local coordinates $(x_i,\xi_a)$ and  $(y_j,\eta_b)$, where $i,j=1,\ldots, n$ and $a,b=1,\ldots, m$. Let in $\mathcal U_1\cap \mathcal U_2$ we have the following transition functions \begin{equation}\label{eq starting smf m=2}
	\begin{split}
		&y_j= F_{j} + \frac12 G_j^{a_1a_2}\xi_{a_1}\xi_{a_2}+\cdots\,\,\, j=1,\ldots, n;\\ 
		&\eta_b = H_{b}^a\xi_a+\cdots, \,\,\, b=1,\ldots, m,
	\end{split}
\end{equation}
\commentMR{It is better to include $\frac{1}{2}$.}
where $F_j,G^{a_1a_2}_j,H_{i}^b$ are (holomorphic) functions depending only on even coordinates $(x_i)$.  

\commentMR{We consider this not only for $z= x_i$, $\xi_a$. Let us work in this generality. It will be shorter.} 
{\newMR If $D: \mcO \to \mcO$ is a differential operator let $\rd{D}: \mcO \to \mcF$ denotes the composition of $D$ with the projection $\mcO \to \mcO/\mcJ  =\mcF$.}

Following Donagi and Witten, see  \cite[Section 2.1]{Witten Atiyah classes}, we define the sheaf $\mathcal D^2_{-} |_{\mcM_0}$ on the underlying manifold $\mcM_0$ as a sheaf locally generated over $\mathcal F$  by 
$$
 \langle 1, \rd{\frac{\partial}{\partial x_i}}, \rd{\frac{\partial}{\partial \xi_j}}, \rd{\frac{\partial^2}{\partial \xi_a \partial \xi_b}} \rangle.
$$ 
  In \cite[Theorem 2.5]{Witten Atiyah classes} it was shown that this definition does not depend on local coordinates. Note that the operators 
  $\frac{\pa}{\pa \xi_a}$ anticommute, ie. $\frac{\partial^2}{\partial \xi_a\partial \xi_b} = - \frac{\partial^2}{\partial \xi_b\partial \xi_a}$.
  We can get transition function for generators of $\mathcal D^2_{-} |_{\mcM_0}$ writing down the transition function for $\frac{\partial}{\partial x_i}$, $\frac{\partial}{\partial \xi_j}$, $\frac{\partial^2}{\partial \xi_a\partial \xi_b}$ and then factorizing them by $\mathcal J$. 
Indeed,
\begin{align*}
    \frac{\partial}{\partial x_i} = \frac{\partial y_j}{\partial x_i} \frac{\partial}{\partial y_j} + \frac{\partial \eta_{\alpha}}{\partial x_i} \frac{\partial}{\partial \eta_{\alpha}};
    \quad 
    \frac{\partial}{\partial \xi_k} = \frac{\partial y_j}{\partial \xi_k} \frac{\partial}{\partial y_j} + \frac{\partial \eta_{\alpha}}{\partial \xi_k} \frac{\partial}{\partial \eta_{\alpha}}.
    \end{align*}
Therefore modulo $\mathcal J$ we have:
\begin{align*}
    \rd{\frac{\partial^2}{\partial \xi_{a_1} \partial \xi_{a_2}}}  =  
    \Big(\frac{\partial \eta_{\alpha}}{\partial \xi_{a_1}}  \frac{\partial^2 (y_j)}{\partial \xi_{a_1}\partial \xi_{a_2}}\Big)_{\red} \rd{\frac{\partial}{\partial y_j}} - \Big(\frac{\partial \eta_{\alpha}}{\partial \xi_{a_1}} \frac{\partial \eta_{\beta}}{\partial \xi_{a_2}}\Big)_{\red} \rd{\frac{\partial^2}{\partial \eta_{\alpha}\partial \eta_{\beta}}}. 
\end{align*}
Compare with \cite[Formula  2.13]{Witten Atiyah classes}.
Using (\ref{eq starting smf m=2}) we can write the transition functions for $\mathcal D^2_{-} |_{\mcM_0}$ explicitly. 
\begin{equation}\label{eq Don Witten}
\begin{split}
    &x_i= F^{-1}_{i} (y_j);\\ 
    &\rd{\frac{\partial}{\partial x_i}} = \frac{\partial F_j}{\partial x_i} \rd{\frac{\partial}{\partial y_j}};\quad \rd{\frac{\partial}{\partial \xi_a}}  = H_{a}^{b}  \rd{\frac{\partial}{\partial \eta_{b}}};\\
   & \rd{\frac{\partial^2}{\partial \xi_{b_1} \partial \xi_{b_2}}}  = -G_j^{b_1b_2}  \rd{\frac{\partial}{\partial y_j}}  - H_{\alpha}^{b_1} H_{\beta}^{b_2}  \rd{\frac{\partial^2}{\partial \eta_{\alpha}\partial \eta_{\beta}}} .
\end{split}
    \end{equation}
\commentMR{I personally prefer to write $\pa_x$ instead of $\frac{\pa}{\pa  x}$. Then the formulas would be shorter. Similarly, the indexes could be better organized: use $i, j, k, .. $ for $x$ and $a, b, ...$ for $\xi$. The minuses would be pluses if we follow another convention at the begging. }

Note that in (\ref{eq Don Witten}) it is more correct to use the index $red$ for even coordinates, for example $(x_i)_{\red}$. However we omit $\red$ for simplicity of notations.  We conclude this subsection with the following important remark. 
\commentMR{D-W factor this module by $\langle 1, \frac{\pa}{\pa \xi}$ so only $\frac{\pa}{\pa x_i}$ and 
$\frac{\partial^2}{\partial \eta_{\alpha}\partial \eta_{\beta}}$ survives. }
\commentMR{Let us stay with as it is.}

\begin{remark}
    Comparing Formulas (\ref{eq starting smf m=2}) and (\ref{eq Don Witten}), we see that Formulas (\ref{eq Don Witten}) contain the whole information about Formulas (\ref{eq starting smf m=2}) modulo  $\mathcal J^3$.
    In other words using Formulas (\ref{eq Don Witten}) we can reconstruct Formulas (\ref{eq starting smf m=2}) modulo $\mathcal J^3$.

    One of purposes of this paper is to develop this observation. 
    This leads to an idea to use the theory of $n$-fold vector bundles and the theory of graded manifolds of degree $n$ to recover a supermanifolds of odd dimension greater than $2$. 
\end{remark}
\commentMR{The rest moved to the next subsection (Lemma)}

\subsection{A geometric interpretation of the Donagi and Witten construction}

In this subsection we give a description of the geometric object with transition functions (\ref{eq Don Witten}). We 
use the theory of double vector bundles and graded manifolds of degree $2$. 
\commentMR{I do not see what was going on in this subsection. 1. Refer to D-W that Atiyah class of a s.e.s. associated with $\mathcal{D}^2_{-}|_{\mcM_0}$ is the obstruction class $\omega_2(\mcM)$. 
2. We want to show that the s.e.s. present in D-W paper (with  $\mathcal{D}^2_{-}|_{\mcM_0}$ in the middle) is just the s.e.s. associated with some a DVB? But this should be replaced with a s.e.s. associated with  a skew-symmetric DVB? Then it makes sense.}


\begin{lemma}  Let $\mathbb{D}$ be a vector bundle over $M$ characterized by the space of sections, $\Sec(\mathbb{D}) =  \mathcal D^2_{-} |_{\mcM_0}$. Then the dual bundle $\mathbb{D}^*$ has transition functions of a skew-symmetric \DVB.  More precisely, there exist a skew-symmetric \DVB  $(D, \sigma)$ such that   the space $\mcA^{\alpha+\beta}_{-}(D)$ coincides with the space of sections of the vector bundle $\mathbb{D}^*$.
\end{lemma}
\begin{proof}
We simply transpose and reverse the formulas \eqref{eq Don Witten}.
\end{proof}
The skew-symmetric \DVB arising from the above lemma will be denoted by $\Vb_2(\mcM)$.
 \commentMR{Some corrections are necessary below in this subsection.}
 
  The Atiyah class of the  exact sequence \eqref{eq exact for DBV}, that is the obstruction class of splitting of this sequence, is an element in $H^1(M;  A^*\otimes  B^* \otimes \ C)$.  It will be also called the Atiyah  class associated with the \DVB $D$ and denoted by $\At(D)$. Note that the Atiyah  class associated with 
  any dual (vertical or horizontal) of $D$ coincides with $\At(D)$.  Conversely 
 if we have the sequence (\ref{eq exact for DBV}) we can reconstruct the double vector bundle $D$. 

Now let $\mcM$ be a supermanifold with transition function (\ref{eq starting smf m=2}) and let $E = \Vb_1(\mcM)$. 
The Atiyah sequence associated with $D = \Vb_2(\mcM)$ (which  is a double vector bundle with the same side bundles $E$ and the core $\T \mcM_0$) is 
$$
 0 \to \T  \mcM_0 \to \widehat{D} \to E \otimes E \to 0.
$$
The obstruction for splitting of this exact sequence is 
$$
\At(\Vb_2(\mcM)) \in H^1(\mcM_0, \T (\mathcal M_0) \otimes  E^*\otimes  E^*).
$$ 
{\newMR Due to the decomposition $E\otimes E  =\Sym^2 E \oplus \bigwedge^2 E$,
the Atiyah class $\At(D)$, where $D =\Vb_2(\mcM)$,  decomposes to $\At_{+}(D) + \At_{-}(D)$, where $\At_{+}(D) \in H^1(\mcM_0, \T (\mathcal M_0) \otimes \Sym^2 E^*)$ and  $\At_{-}(D) \in H^1(M; \T M \otimes \bigwedge^2  E^*)$. We are only interested with $\At_{-}(D)$  because our DVB $\Vb_2(\mcM)$ is a skew-symmetric \DVB, hence $\At_+(D) = 0$.}
\commentMR{The situation is a bit more complicated: The Atiyah class of s.e.s. associated with $D = \Vb_2(\mcM)$ decomposes to $\At_{+}(D) + \At_{-}(D)$. However,  we are only interested with  $\At_{-}(D) \in H^1(M; \T M \otimes \bigwedge^2 E^2)$ because our DVB $\Vb_2$ is skew-symmetric. Write this is more details. }

{\newMR The class $\At_{-}(\mcM)$ is here 
the same as  the Atiyah class of the sequence \eqref{e:ses-skewDVB} associated with the skew-symmetric \DVB $(D, \sigma)$.  It also coincides with  the first obstruction class to splitting  of the supermanifold $\mcM$ as in \cite[Section 2]{Witten Atiyah classes}. } \commentMR{The last sequence I did not check.}


Later we will give a geometric interpretation of this fact.

\subsection{A modification of the Donagi--Witten construction}
In this subsection we suggest a different way how to obtain a double vector bundle with obstruction the class $\omega_2$. First of all to write Formulas (\ref{eq Don Witten}) we need the transition functions  (\ref{eq starting smf m=2}) and their inverse. To avoid this inconvenience we can use differential forms instead of differential operators. 

Let $\mcM$ be a supermanifold as above. Consider as above two charts $\mathcal U_1$ and $\mathcal U_2$ with non-empty intersection on $\mcM$ with local coordinates $(x_i,\xi_a)$ and  $(y_j,\eta_b)$ and  transition functions (\ref{eq starting smf m=2}). Further consider the antitangent {\newMR bundles} $\oT(\mathcal M)$ of $\mcM$ and two charts $\oT(\mathcal U_1)$ and $\oT(\mathcal U_2)$ with standard coordinates $(x_i,\xi_a, \dd x_i,\dd \xi_a)$ and  $(y_j,\eta_b,\dd y_j,\dd\eta_b)$, respectively. Thus $x_i$, $\dd\xi_a$ are even local coordinates in $\oT(\mathcal U_1)$, while $\xi_a$,$\dd x_i)$  are odd ones. 
In $\oT(\mathcal U_1)\cap\oT(\mathcal U_2)$ we get the following transition functions.
\begin{align*}
	&	y_j= F_{j} +G_j^{a_1a_2}\xi_{a_1}\xi_{a_2}+\cdots,\quad 
	\eta_a = H_{a}^b\xi_b+\cdots;\\
	& \dd y_j = \sum_{b=1}^n(F_{j})_{x_b}\dd x_b + \sum_{b=1}^n(G_j^{a_1a_2})_{x_b}\dd x_b\xi_1\xi_2 + G_j^{a_1a_2}(\dd(\xi_{a_1}) \xi_{a_2} - \xi_{a_1}\dd(\xi_{a_2}))+\cdots;\\
	&\dd \eta_a = \sum_{b=1}^n (H_{a}^i)_{x_b}\dd x_b \xi_j +  H_{a}^b \dd \xi_b+\cdots.\\
\end{align*}
Here we denoted by $(F_{j})_{x_b}$ the derivation of $F_{j}$ by $x_b$.  Now we apply the functor split $\gr$ to $\oT(\mathcal M)$. In local coordinates this means that we factorize our transition functions by all terms that contain more than one odd variable. 
 Then $\gr \oT(\mathcal M) $ has the following transition functions
\begin{equation}\label{eq transition function for grT(M)}
\begin{split}
    &	y_j= F_{j}(x);\\ 
	&\eta_a =  H_{a}^b\xi_b;\\
	& \dd y_j = \sum_{b=1}^n(F_{j})_{x_b}\dd x_b  + G_j^{a_1a_2}(\dd(\xi_{a_1}) \xi_{a_2} - \xi_{a_1}\dd(\xi_{a_2}));\\
	&\dd \eta_a =   H_{a}^b \dd \xi_b.
\end{split}
\end{equation}

If we compare Formulas (\ref{eq transition function for grT(M)}) with \cite[Section 2.2, Formulas (9)-(12)]{Voronov}, we see that a manifold with such transition functions is a double vector bundle, which we denote by $\mathbb D$. We can see  $\gr \oT(\mathcal M)$ as a graded manifold of type $\Delta= \{0, \alpha, \beta, \alpha+\beta\}$, where $\alpha$ is odd {\newMR (the weight of $\xi_a$)} and  $\beta$ is even {\newMR (the weight of $\dd \xi$)}. 
Formulas (\ref{eq transition function for grT(M)}) shows that all weights are well-defined.
The double vector bundle
\begin{equation}\label{eq our DVB}
    \begin{tikzcd}
\mathbb D' \arrow[r, ] \arrow[d]
&  E[\bar 1] \arrow[d,  ] \\
 E \arrow[r,  ]
& \mcM_0
\end{tikzcd}
\end{equation}
is only slightly different from the double vector bundle $\Vb_2(\mcM)$. 
In our case one side bundle is pure even. 

\begin{remark}
    By definition the composition of functors $\gr\circ \oT$ is a functor from the category of supermanifolds to the category of (split) supermanifolds. However above implies that we can see the image of $\gr\circ \oT$ as 
    the category of double vector bundles {\newMR (with some additional structure)}. For simplicity we will use the same notation $\gr\circ \oT$ meaning a functor from the category of supermanifolds to the category of double vector bundles. Similarly in next sections we will consider the functor $\gr\circ \oT^{(n)}$ as a functor from the category of supermanifolds to the category of $n$-fold vector bundles.
\end{remark}

Now we can go further and give an interpretation of $\omega_2$ as the obstruction class of splitting of a  graded manifold of degree $2$.  First of all let us change that parity of the side bundle $\mathbb E$. (In other words we apply the functor parity change, see \cite{Voronov}.) To do this we need to rewrite (\ref{eq transition function for grT(M)}) in the following form
\begin{equation}\label{eq transition function for pi grT(M)}
\begin{split}
    &	y_j= F_{j};\\ 
	&\eta_a =  H_{a}^b\xi_b;\\
	& \dd y_j = \sum_{b=1}^n(F_{j})_{x_b}\dd x_b  + G_j^{a_1a_2}(\xi_{a_2}\dd(\xi_{a_1}) - \xi_{a_1}\dd(\xi_{a_2}));\\
	&\dd \eta_a =   H_{a}^b \dd \xi_b.
\end{split}
\end{equation}
and to change the parity of the weight $\beta$. 

Now we use a result obtained by \cite{JL,Fernando} for double vector bundles and by \cite{Grab,Vish} for $n$-fold vector bundles. Later we will call this step "to apply the functor inverse". In more details in \cite{JL,Fernando} a functor was constructed from the category of graded manifolds of degree $2$ to the category of double vector bundles with some additional structures. In \cite{Grab,Vish} an analogue of this result was obtained for graded manifolds of degree $n$ and  for graded manifolds of type $\Delta$. (Note that in all these papers \cite{JL,Fernando,Grab,Vish} the categories of double vector bundles with additional structures are different. In this paper we follow approaches of \cite{Grab,Vish}.) In \cite{JL,Fernando} it was shown that this functor is an equivalence of the category of double vector bundles with some additional structures and the category of graded manifolds of degree $2$. In  \cite{Grab,Vish} it was shown that this functor is an equivalence of the category of $n$-fold vector bundles with some additional structures and the category of graded manifolds of degree $n$ or of type $\Delta$. The inversion of this functor, that is the functor from  the category of $n$-fold vector bundles with some additional structures to the category of graded manifolds of degree $n$, we call the {\it functor inverse}. We will denote the functor inverse by $\iota$.

In terms of local coordinates "to apply the functor inverse" means that we identify $\xi_a$ with $\dd\xi_a$ in (\ref{eq transition function for pi grT(M)}). We get
\begin{equation}\label{eq transition function for iota pi grT(M)}
\begin{split}
    &	y_j= F_{j};\\ 
	&\theta_a =  H_{a}^i\zeta_i;\\
	& t_j = \sum_{b=1}^n(F_{j})_{x_b}z_b  + 2G_j^{a_1a_2}\zeta_{a_1}\zeta_{a_2}.
\end{split}
\end{equation}
We obtained transition functions of a graded manifold $\mathcal N$ of degree $2$, which we will denote also by $\F_2(\mcM)$.  We assign the following weights to our local coordinates: $x_i$ (weight $0$); $\zeta_j$ (weight $\alpha$); $z_s$ (weight $2\alpha$). In other words, the transition functions (\ref{eq transition function for iota pi grT(M)})  defines a graded manifold of type $\Delta=\{ 0,\alpha,2\alpha\}$, where $\alpha$ is odd. Note that we can remove the coefficient $2$ in (\ref{eq transition function for iota pi grT(M)}). Indeed, it is enough to replace $\zeta_i$ by $\frac{1}{\sqrt{2}}\zeta'_i$ in any chart.

\begin{remark}
    Let us give another explanation of the procedure "to apply the functor inverse". Consider two graded domain $\mathcal V_1$ and $\mathcal V_2$ with graded coordinates $(x_i,\zeta_j, z_i)$ and $(y_i,\theta_j, t_i)$, respectively, and with weights as above. Define transition functions $\mathcal V_1 \to \mathcal V_2$ by (\ref{eq transition function for iota pi grT(M)}). (Let us omit the coefficient $2$.) Now we apply the tangent functor $\oT$ to $\mathcal V_1$ and $\mathcal V_2$ and to the morphism (\ref{eq transition function for iota pi grT(M)}) and factorize the result by the sheaf of ideals locally generated by $\dd x_i$ (or by $\dd y_i$). We get
    \begin{align*}
    &	y_j= F_{j};\\ 
	&\theta_a =  H_{a}^i\zeta_i;\quad \dd \theta_a =   H_{a}^i \dd \zeta_i;\\
	& \dd t_j = \sum_{b=1}^n(F_{j})_{x_b}\dd z_b  + G_j^{a_1a_2}(\zeta_{a_2}\dd(\zeta_{a_1}) - \zeta_{a_1}\dd(\zeta_{a_2})).
    \end{align*}
    We obtain Formulas (\ref{eq transition function for pi grT(M)}) up to appropriate change of variables.  The inversion of this procedure is called  "to apply the functor inverse".  \end{remark}

The structure sheaf $\mathcal O_{\mathcal N}=\bigoplus_{q\geq 0}(\mathcal O_{\mathcal N})_q$ of $\mathcal N$ is $\Z$-graded. Clearly $(\mathcal O_{\mathcal N})_1\cdot (\mathcal O_{\mathcal N})_1\subset (\mathcal O_{\mathcal N})_2$. Therefore we can assign to $\mathcal N$ the following  exact sequence
\begin{align*}
    0\to (\mathcal O_{\mathcal N})_1\cdot (\mathcal O_{\mathcal N})_1\to (\mathcal O_{\mathcal N})_2 \to (\mathcal O_{\mathcal N})_2/ (\mathcal O_{\mathcal N})_1\cdot (\mathcal O_{\mathcal N})_1\to 0. 
\end{align*}
The Atiyah class of this sequence is represented by a cocycle $(G_J^{a_1a_2})$ and coincides with the obstruction class of the double vector bundle (\ref{eq transition function for pi grT(M)}). It is called the {\it obstruction class of the graded manifold $\F_2(\mcM)$ of degree $2$}.
\commentMR{For existence of the functor inverse it is better to refer to our previous papers.}

Let us summarize our results in the following theorem.

\begin{theorem}[Main theorem about the first obstruction class]
    The first obstruction class $w_2$ for a supermanifold $\mcM$ coincides with the obstruction class $\At_{-}(\Vb_2(\mcM))$ of the skew-symetric double vector bundle $\Vb_2(\mcM)$ 
    and with the obstruction class of the graded manifold $\F_2(\mcM)$ of degree $2$.
\end{theorem}
\commentMR{$\At_{+}(D)$ where $D$ is a skew-symmetric \DVB is always zero, since the associated cocycle is zero. Hence $\At_{-}(D) = \At(D)$ so we can return to $\At$ as well in the formulation of the theorem.}

Note that the map $\mcM \mapsto \F_{2}(\mcM)$ is a functor  from the category of supermanifolds of odd dimension $2$ to the category of graded manifolds of degree $2$. Indeed,  $\F_2$ is a composition of four functors: $\oT$, $\gr$, $\pi$ and $\iota$. This functor defines an equivalence of the category of supermanifolds of odd dimension $2$ and the category of graded manifolds of degree $2$ with the following additional condition. If $\mathcal N=(\mathcal N_0, \mathcal O_{\mathcal N})$ is a graded manifold of degree $2$, we additionally assume that the locally free sheaf $\mathcal E:= (\mathcal O_{\mathcal N})_2/(\mathcal O_{\mathcal N})_1\cdot (\mathcal O_{\mathcal N})_1$ is isomorphic to the sheaf of sections of $\T[\bar 1](\mathcal N_0)$. In this case we can choose local coordinates in the form (\ref{eq transition function for iota pi grT(M)}). Using (\ref{eq transition function for iota pi grT(M)}) we can write transition function for a supermanifold of odd dimension $2$
\begin{align*}
     &	y_j= F_{j}+G_j^{a_1a_2}\zeta_{a_1}\zeta_{a_2} ;\\ 
	&\theta_a =  H_{a}^i\zeta_i.
\end{align*}
These transition functions satisfy the cocycle condition. Indeed, consider three charts $\mathcal U_1, \mathcal U_2$ and $\mathcal U_3$ with $\mathcal U_1\cap \mathcal U_2\cap \mathcal U_3\ne \emptyset$. Denote by $T_{ij}$ the transition function  $T_{ij}: \mathcal U_i\to \mathcal U_j$. Consider the following composition of maps
$$
T_{31}\circ T_{23}\circ  T_{12} =:  R.
$$
Since $\F_2$ is a functor, we get
$$
\F_2(T_{31})\circ \F_2(T_{23})\circ \F_2( T_{12}) = \F_2( R).
$$
The composition $\F_2(T_{31})\circ \F_2(T_{23})\circ \F_2( T_{12})$ is equal to $id$, since the cocycle condition for the graded manifold $\mathcal N$ holds true. Therefore $\F_2( R)=id$. The morphism $R:\mathcal U_1\to \mathcal U_1$ is completely defined by its image $\F_2( R)=id$. Therefore, $R=id$. 

Denote the category of supermanifold of odd dimension $2$ by $\mathrm{Smf_2}$ and the category of graded manifolds of degree $2$ and of odd dimension $2$ with the additional condition for $ (\mathcal O_{\mathcal N})_2/(\mathcal O_{\mathcal N})_1\cdot (\mathcal O_{\mathcal N})_1$ as above, by $\mathrm{GrMan_2T}$. Now we can summarize our results in the following theorem. 

\begin{theorem}\label{theor cat Smf2 and GrMan2T are equivalent}
    The categories $\mathrm{Smf_2}$ and $\mathrm{GrMan_2T}$ are equivalent. 
\end{theorem}

In \cite{RV} we generalize this theorem to the category of supermanifolds of any odd dimension.

\begin{remark}
    If our supermanifold $\mcM$ has the odd dimension $m> 2$,  we still can repeat the procedure above. Therefore the functor $\F_2$ is a functor from the category of supermanifolds to the category of graded manifold of degree $2$. However for any dimension the functor $\F_2$ is not an embedding. In this case from a graded manifold of degree $2$ we can recover a supermanifold module $\mathcal J^3$.  
\end{remark}

Summing up, in this section we showed that the first obstruction class to splitting of the supermanifold in the sense of \cite[Section 2]{Witten Atiyah classes} coincides with the obstruction class of the splitting of the double vector bundle $\gr \oT(\mathcal M) $ and with the obstruction class of the splitting of the graded manifold $\F_2(\mathcal M) $. A general splitting theory for supermanifolds of odd dimension $m$ based on splitting of the corresponding $m$-fold vector bundles and graded manifolds of degree $m$ will be developed in our oncoming paper.

\section{A generalization of the Donagi--Witten construction}\label{sec gen Don Witten}

In this section we give a construction of functors $\F_n$, where $n=2,3,\ldots, \infty$,  from the category of supermanifolds to the category of graded manifolds of degree $n$. Due to the size of this paper and technical difficulty of the proof of the main result, the existence of the functor inverse $\iota$ for any supermanifold of the form $\gr\circ \oT(\mathcal M)$, where $\mcM$ is a supermanifold, we leave details of this proof to Part~$II$ of this paper. Here we give only the idea of the proof.

The functor $\F_n$ is again a composition of four functors: the $(n-1)$-times (or infinity many times for $n=\infty$) iterated antitangent  functor $\oT^{(n-1)}: = \oT\circ \cdots \circ \oT$, the functor split $\gr$, the functor parity change $\pi$ and the functor inverse $\iota$. If $\mcM$ is a supermanifold, then $\oT^{(n-1)}(\mcM)$, $\gr \circ \oT^{(n-1)}(\mcM)$ are already defined.  
Denote by $\Delta$ the maximal multiplicity free weight system generated by an odd weight $\alpha$ and by even weights $\beta_1,\ldots, \beta_{n-1}$, see Definition \ref{def weight system}. We need the following propositions.

\begin{proposition}\label{prop gr T is n-fold}
    The supermanifold $\gr \circ \oT^{(n-1)}(\mcM)$ is an $n$-fold vector bundle of type $\Delta$, where $\Delta$ is the maximal multiplicity free system generated by an odd weight $\alpha$ and by even weights $\beta_1,\ldots, \beta_{n-1}$. 
\end{proposition}
\begin{proof}
    Let $\dd_i$, where $i=1,\cdots, n-1$, be the iterated de Rham differentials. They are vector fields in the structure sheaf of $\oT^{(n-1)}(\mcM)$. Denote by $\dd_I$, where $I=(i_1,\ldots, i_k)$, the composition $\dd_{i_1}\circ \ldots \circ \dd_{i_k}$. Let $\mathcal I$ be the ideal generated by odd elements in $\mathcal O_{\gr \circ \oT^{(n-1)}(\mcM)}$ and let us choose a chart $\mathcal U$ on $\mcM$ with local coordinates $(x_a,\xi_b)$. Then the corresponding standard local coordinates in the chart $\oT^{(n-1)}(\mathcal U)$ are $(\dd_I(x_a),\dd_J(\xi_b))_{C (I),C(J) \leq n-1}$, where $C (I)$ is the cardinality of $I$. Hence  the corresponding local coordinates in the chart $\gr \circ\oT^{(n-1)}(\mathcal U)$ are $(\dd_I(x_a)+ \mathcal I^2,\dd_J(\xi_b)+ \mathcal I^2)_{C (I),C(J) \leq n-1}$. We assign the weight $\alpha+ \beta_{i_i} +\cdots+ \beta_{i_k}$ to $\dd_I(Z)+ \mathcal I^2$, where $Z\in \{ x_a,\xi_b\}$, if $\dd_I(Z)$ is odd and the weight $\beta_{i_i} +\cdots+ \beta_{i_k}$ to $\dd_I(Z)+ \mathcal I^2$ if $\dd_I(Z)$ is even. 
    
    It remains to prove that transition functions in $\gr \circ \oT^{(n-1)}(\mcM)$ preserve our weights. Let $\mathcal U'$ be another chart on $\mcM$ with coordinates $(x'_a,\xi'_b)$ such that $\mathcal U\cap \mathcal U'\ne \emptyset$. Let 
    $$
    x_a'= \sum_{C (K)=2k} F_{K}(x) \xi_K
    $$
  be the expression of $x_a'$ in coordinates of $\mathcal U$. Then $\dd_I(x'_a)+ \mathcal I^2$ in coordinates of $\gr \circ\oT^{(n-1)}(\mathcal U)$ is a sum of monomials in the following form. 
\begin{equation}\label{eq weights}
    H_{I,K}(x) \dd_{I_1} (x_{a_1}) \cdots \dd_{I_p} (x_{a_p})  \dd_{I_{p+1}} (\xi_{b_1}) \cdots \dd_{I_q} (\xi_{b_q}) + \mathcal I^2,
\end{equation}
where $H_{I,K}(x)$ is a certain derivative of $F_K(x)$, which has weight $0$, and  $(I_1,\ldots,I_q)$ is a decomposition of the sequence $I$ into $q$ parts. We see that the weights of $\dd_I(x'_a)+ \mathcal I^2$ and of (\ref{eq weights}) coincide. Similarly for $\dd_I(\xi'_b)+ \mathcal I^2$. This completes the proof. 
\end{proof}

\begin{proposition}\label{prop gr T^n is embedding}
The functor $\gr \circ \oT^{(n-1)}$ is an embedding of the category of supermanifolds of odd dimension $n$ into the category of $n$-fold vector bundles of type $\Delta$.  
\end{proposition}
\begin{proof}

Let $\mathcal N$ be an $n$-fold vector bundle. Let us show that if a preimage $\mcM$ of $\mathcal N$ exists, it is unique.   The idea of the proof of similar to the idea of the proof of Theorem \ref{theor cat Smf2 and GrMan2T are equivalent}. Assume that there exists another supermanifold $\mcM'$ with $\gr \circ \oT^{(n-1)}(\mcM)= \gr \circ \oT^{(n-1)}(\mcM')$.  By construction we have $\mcM_0=\mcM'_0= \mathcal N_0$. Let us choose an atlas $\{\mathcal U_i\}$ on $\mcM$ and an atlas $\{\mathcal U'_i\}$ on $\mcM'$ such that  $(\mathcal U_i)_0=(\mathcal U'_i)_0$. Clearly such atlases there exist, since it is sufficient to choose any atlas on $\mcM_0$ of Stein domains. Now consider two charts $\mathcal U'_1$ and $\mathcal U'_2$ with $\mathcal U'_1\cap \mathcal U'_2\ne \emptyset$ on $\mcM'$. Denote by $T'$ the transition function  $T': \mathcal U'_1\to \mathcal U'_2$ and by  $T$ the transition function  $T: \mathcal U_1\to \mathcal U_2$. By our assumption $\gr \circ \oT^{(n-1)} (T)= \gr \circ \oT^{(n-1)}(T')$. But this implies that $T=T'$. The proof is complete. 
\end{proof}

\begin{remark}
The question when the functor $\gr \circ \oT^{(n-1)}$ possesses a preimage is treated in \cite{RV}.     
\end{remark}

Now we can use results of \cite{Grab} and \cite{Vish} to define the functor inverse. Note that these two results lead to two different approaches to the problem. Let us start with a description of results  \cite{Grab} or \cite{Vish}. In \cite{Grab} a functor $\G_n$ was constructed from the category of graded manifolds of degree $n$ to the category of $n$-fold symmetric vector bundles. In \cite{Vish} a functor $\V_n$ was constructed from the category of graded manifolds of degree $n$ to the category of $n$-fold vector bundles with $n-1$ odd commutative homological vector fields. In both cases it was proven that these functors are equivalence of categories. In Section \ref{sec Witten} we saw that due to the fact that  the first obstruction class $\omega_2$ is an element in 
$$
H^1(\mcM_0,\T[\bar 1] (\mathcal M_0) \otimes \bigwedge^2\mathbb E^*)
$$
we can construct a graded manifold of degree $2$ with transition functions (\ref{eq transition function for iota pi grT(M)}) from a double vector bundle given by (\ref{eq transition function for pi grT(M)}). 
In terms of \cite{Grab} this means that the double vector bundle with the obstruction class $\omega_2$ is symmetric and  $\pi\circ \gr \circ \oT(\mcM)$ is in the image of the functor $\G_2$ constructed in \cite{Grab}. Our functor inverse $\iota$ is the inversion of the equivalence $\G_2$. More generally we have
\begin{theorem}\label{ther G_n}
    For a supermanifold $\mcM$ of odd dimension $n$ the image $\pi\circ \gr \circ \oT^{(n-1)}(\mcM)$ is a symmetric $n$-fold vector bundle. Therefore there exists a graded manifold $\mathcal N$ of degree $n$ such that $\G_n(\mathcal N)\simeq \pi\circ \gr \circ \oT^{(n-1)}(\mcM)$. In other words the functor inverse $\iota= \G_n^{-1}$ is defined on the image of the functor $\pi\circ \gr \circ \oT^{(n-1)}$. 
\end{theorem}

Let us describe another approach based on results of  \cite{Vish}.  The graded super\-manifold $\gr \circ \oT(\mcM)$ possesses a natural odd homological vector field. Indeed, let $\dd_{R}$ be the exterior derivative (de Rham defferential) of the supermanifold $\mcM$. Clearly $\dd_{R}$ is an odd homological vector field in the structure sheaf $\mathcal O_{\oT(\mcM)}$ of $\oT(\mcM)$. Denote by $\mathcal I$ the sheaf of ideals locally generated by all odd variables of $\mathcal O_{\oT(\mcM)}$. Then we have
$$
\dd_{R} (\mathcal I) \subset \mathcal O_{\oT(\mcM)}\quad \text{and} \quad \dd_{R} (\mathcal I^2) \subset \mathcal I. 
$$
Therefore we have an induced operator in the structure sheaf of $\gr \circ \oT(\mcM)$
$$
\mathcal O_{\gr \circ \oT(\mcM)} = \gr \mathcal O_{\oT(\mcM)} = \bigoplus_{p\geq 0} \mathcal I^p/\mathcal I^{p+1}. 
$$
 Clearly this induced operator is odd and homological. Hence the double vector bundle  $\gr \circ \oT(\mcM)$ is a double vector bundle with a homological vector field. Therefore by \cite{Vish} there exists a graded manifold $\V_2^{-1}(\gr \circ \oT(\mcM))$ of degree $2$. This procedute can be generalized.  Now we can formulate a general result.
\begin{theorem}\label{ther V_n}
     For a supermanifold $\mcM$ of odd dimension $n$ the image $\gr \circ \oT^{(n-1)}(\mcM)$ is an $n$-fold vector bundle with $n-1$ commuting odd homological vector fields. Therefore there exists a graded manifold $\mathcal N$ of degree $n$ such that $\V_n(\mathcal N)\simeq \gr \circ \oT^{(n-1)}(\mcM)$. In other words the functor $\iota'= \V_n^{-1}$ is defined on the image of $\gr \circ \oT^{(n-1)}$. 
\end{theorem}

 The idea of the proof is simple: the $(n-1)$-times iterated tangent functor leads to $n-1$ commuting de Rham differentials, which induce $(n-1)$  commuting odd homological vector fields on $\gr \circ \oT^{(n-1)}(\mcM)$. We leave details of the proof to Part II of this paper. 
Note that in this case the functor parity change $\pi$ is included in the functor $\V_n$, therefore we do not need to apply it explicitly. We conclude this section with the following proposition
\begin{proposition}
    The functors $\iota \circ \pi \circ \gr \circ \oT^{(n-1)}$ and $\iota' \circ \gr \circ \oT^{(n-1)}$ are embeddings of the category of supermanifolds of odd dimension $n$ to the category of graded manifolds of degree $n$. 
\end{proposition}
\begin{proof}
    The proposition follows from Proposition \ref{prop gr T^n is embedding}, from the fact that $\pi$ is an equivalence of categories of $n$-fold vector bundles and from Theorems \ref{ther G_n} and \ref{ther V_n}.
\end{proof}

The inverse limit of functors $\F_n$ is denoted by $\F_{\infty}$. That is if $\mcM$ is a supermanifold,  we put $\F_{\infty} (\mcM):= \varprojlim \F_{n} (\mcM)$ and the same for morphisms.

  \section{Donagi--Witten functor for Lie superalgebras and\\ Lie supergroups}

In this section, we adapt the results of Section \ref{sec Witten} to the case of Lie superagebras and Lie supergroups. We show that there is an analogue of  the functor $\F_2$ in the category of Lie superalgebras.  We denote by $\SLie$ the category of Lie superalgebras, by $\GrLie_n$ the category of $\Z$-graded Lie superalgebras of type $\{0,1,\ldots, n\}$ with $|1|=\bar 1$ and by $\GrLie_{\infty}$ the  category of non-negatively $\mathbb Z$-graded Lie superalgebras.

\subsection{Donagi and Witten construction for Lie superalgebras}\label{sec DW for superalgebras}

In this section we will construct a functor  $\F'_2: \SLie \to \GrLie_2$.
In next sections we will use this construction to obtain a functor from the category of Lie superalgebras $\SLie$ to the category $\GrLie_{\infty}$.
As in \ref{sec Witten} the functor $\mathrm F'_2$ is a composition of four functors, which we denote by  $\oT'$ ({\it tangent}), $\mathrm{gr}'$ ({\it split}), $\pi'$ ({\it parity change}) and $\iota'$ ({\it inverse}). (To distinguish the category of Lie superalgebras in the case  we will use superscript. That is $\oT'$ instead of $\oT$ and so on.)
 Let $\mathfrak g= \mathfrak g_{\bar 0} \oplus \mathfrak g_{\bar 1}$ be a Lie superalgebra. Recall that to define a Lie superalgebra we need to define a Lie algebra $\mathfrak g_{\bar 0}$, a $\mathfrak g_{\bar 0}$-module $\mathfrak g_{\bar 1}$ and a symmetric  $\g_{\bar{0}}$-module map $ [\cdot, \cdot]:\mathfrak g_{\bar 1}\otimes \mathfrak g_{\bar 1} \to \mathfrak g_{\bar 0}$
 such that   $[[Y, Y], Y]=0$, for any $Y\in \g_{\bar{1}}$.

{\it Tangent functor $\oT'$.} Since a Lie superalgebra $\mathfrak g=
\mathfrak g_{\bar 0} \oplus \mathfrak g_{\bar 1}$ is a vector superspace, its antitangent bundle is the following linear superspace
$$
\oT'(\mathfrak g) =  \mathfrak g \oplus \mathrm{d}(\mathfrak g).
$$
Here $\mathrm{d} \g$ denote a copy of $\mathfrak g$ with reversed parity.  In more details, $V=\mathfrak g \oplus \mathrm{d}(\mathfrak g)$ is a linear superspace with the underlying space $V_{\bar 0}=\mathfrak g_{\bar 0}\oplus \dd \mathfrak g_{\bar 1}$ and with the structure sheaf $\mathcal F_{V_{\bar 0}}\otimes S^*(\mathfrak g^*_{\bar 1}\oplus \dd \mathfrak g^*_{\bar 0})$. If $X\in \mathfrak g$, sometimes we will denote by $\mathrm{d}(X)$ the corresponding element in $\mathrm{d}(\mathfrak g)$. Further the Lie superalgebra structure on $\oT'(\mathfrak g)$ is defined as follows. The Lie bracket on the vector subspece $\mathfrak g\oplus \{0\}$ coincides with the Lie bracket of $\mathfrak g$, further $\mathrm{d}(\mathfrak g)$ is a $\mathfrak g$-module, where the action of $\mathfrak g$ coincide with the adjoint action up to sign
$$
[X,\mathrm{d}Y] : = (-1)^{|X|}\mathrm{d} ([X,Y]).
$$
and the product $[\mathrm{d}(\mathfrak g), \mathrm{d}(\mathfrak g)]=0$ is trivial. 

\begin{remark}\label{rem weights of g}
	The Lie superalgebra $\oT'(\mathfrak g) $ sometimes is called in the literature  a {\it Takiff superalgebra}, named after the author of \cite{Tak}. This Lie superalgebra is a Lie superalgebra in the form $\mathfrak g\otimes_{\mathbb K}\mathbb K[\tau]$, where $\tau$ is an odd element and $\mathbb K[\tau]$ are all polynomials in $\tau$. Summing up, in our paper $\oT'$ is derived from both ``Takiff'' and ``antitangent''.
\end{remark}

\medskip

{\it Functor split $\mathrm{gr}'$}. The functor $\mathrm{gr}'$ is defined as follows. For a Lie superalgebra $\mathfrak g$ we put $\mathrm{gr}'(\mathfrak g): = \mathfrak g'$, where $\mathfrak g'$ is obtained from the Lie superalgebra $\mathfrak g$ putting $[\mathfrak g'_{\bar 1}, \mathfrak g'_{\bar 1}] =0.$ In  \cite[Theorem 3]{Vish_funk}, see also Therem \ref{theor split supergroup} below, it was shown that $\Lie(\gr\mcG) = \gr'(\Lie \mcG)$ for any Lie supergroup $\mcG$. In other words $\mathfrak g'$ is the Lie superalgebra of $\gr(\mcG)$, where $\mcG$ is a Lie supergroup with the Lie superalgebra $\mathfrak g$. Clearly $\gr'$ is defined on morphisms as well.  Let us compute the composition of the functors $\mathrm{gr}'\circ \oT'$.

\begin{lemma}\label{lem mult in gr T(g)}
	Let $\mathfrak g$ be a Lie superalgebra. Then the Lie superalgebra $\mathrm{gr}' (\oT'(\mathfrak g))$ is equal to $\oT'(\mathfrak g)$
	as $\mathfrak g_{\bar 0}$-modules. And the following holds true for the Lie superalgebra multiplication
	\begin{align*}
	[\mathfrak g_{\bar 1}, \mathfrak g_{\bar 1}] =0, \quad
	[\mathfrak g_{\bar 1}, \mathrm{d}(\mathfrak g_{\bar 0})] =0, \quad
	[\mathrm{d}(\mathfrak g), \mathrm{d}(\mathfrak g)] =0.
	\end{align*}
	and $[Y_1,\mathrm{d}(Y_2)] = -  \mathrm{d}([Y_1,Y_2])$ for $Y_1,Y_2\in \mathfrak g_{\bar 1}$.
\end{lemma}

\begin{remark} {\newMR Note that the Lie algebra $\gr \oT (\g)$ keeps all the information on the bracket on $\g$. }
    We can see the Lie superalgebra $\mathfrak a=\mathrm{gr}' (\oT'(\mathfrak g))$ as a $\Z\times\Z$-graded Lie superalgebra of type $\{0,\alpha,\beta, \alpha+\beta\}$, where $\alpha$ is odd and $\beta$ is even. Indeed, we put
    $$
    \mathfrak a_0:= \mathfrak g_{\bar 0}, \quad \mathfrak a_{\alpha}:= \mathfrak g_{\bar 1}, \quad\mathfrak a_{\beta}:= \dd\mathfrak g_{\bar 1}, \quad\mathfrak a_{\alpha+\beta}:= \dd\mathfrak g_{\bar 0}. 
    $$
    Lemma \ref{lem mult in gr T(g)} implies that the multiplication in $\mathfrak a$ is $\Z\times\Z$-graded. 
\end{remark}

{\it Functor parity change $\pi'$}. In previous sections we reminded how to define the functor $\pi$ for double vector bundles. In this section we show that on the parity reversed double vector bundle $\pi \circ \gr\circ \oT(\mcG)$, where $\mcG$ is a Lie supergroup, a Lie supergroup structure can be defined. We start with the case of Lie superalgebras. 

We define the functor $\pi'$ on the image of the composition of functors $\mathrm{gr}' \circ \oT'$ as follows. The  Lie superalgebra $\mathrm{gr}' (\oT'(\mathfrak g))$ possesses a parity reversion of the weight $\beta$. 
Indeed, by definition the Lie superalgebra $\pi'(\mathrm{gr}' (\oT'(\mathfrak g)))$ is equal to $\mathrm{gr}' (\oT'(\mathfrak g))$ as $\mathfrak g_{\bar 0}$-modules, but now we assume that elements $X$ and $\mathrm{d}(X)$ have the same parities. In other words we assume that $\mathrm{d}$ is even. More precisely, let $\mathfrak a = \mathrm{gr}' (\oT'(\mathfrak g))$. Denote by $\mathfrak h$ the vector superspace $\pi'(\mathrm{gr}' (\oT'(\mathfrak g)))$. We have 
$$
\mathfrak h = \bigoplus_{\delta\in \Delta} \mathfrak h_{\delta},
$$
where $\Delta = \{0,\alpha,\beta, \alpha+\beta\}$ with $|\alpha|=|\beta|=\bar 1$. Let us define a Lie superalgebra structure on $\mathfrak h$ of type $\Delta$. We denote by $[\,,]_{\mathfrak a}$ the Lie bracket on $\mathfrak a$ and by $[\,,]_{\mathfrak h}$ the Lie bracket on $\mathfrak h$.

\begin{proposition}\label{prop pi gr T (g)Lie superalgebra}
The Lie superalgebra structure on $\mathfrak h$ is defines by the following data. 
 \begin{itemize}
            \item We set $[X, Y]_{\h} = [X, Y]_{\mathfrak a}$ for all homogeneous $X,Y$, except for the case
     $X\in \h_{\alpha}$ and $Y\in \h_{\beta}$.
     
  \item For  $X\in \h_{\alpha}$ and $Y\in \h_{\beta}$ we set $[X, Y]_{\h} := [X, Y]_{\aaa} = - [Y, X]_{\aaa} = :[Y, X]_{\h}$.
     \end{itemize}
  \end{proposition}
\begin{proof}
  {\it Step 1.}  Clearly $\h_{\bar{0}} = \h_{0}\oplus \h_{\alpha+\beta}$ is a Lie algebra, namely it is a semidirect product of the Lie algebra $\aaa_{0}$ and  $\aaa_{0}$-module $\aaa_{\alpha+\beta}$.

     {\it Step 2.} $\h_{\bar{1}}$ is a $\h_{\bar{0}}$-module. Clearly, $\h_{\bar{1}}$ is a $\h_{0}$-module and $\h_{\alpha+\beta}$ acts trivially on $\h_{\bar{1}}$.

   {\it Step 3.} Note that the map $\h_{\bar{1}} \otimes \h_{\bar{1}}\to \h_{\bar{0}}$ induced by $[\,,]_{\h}$ is an $\h_{\bar{0}}$-module map. It remains only to check Jacobi identity for homogeneous $X,Y,Z\in \h_{\bar{1}} = \h_{\alpha}\oplus \h_{\beta}$. However the product of any three elements of this form is $0$. 
\end{proof}

Later we will give another proof of this result. It is easy to check that $\pi'$ is defined on the morphisms of the form $\gr'\circ\oT'(\phi)$, where $\phi:\mathfrak g\to \mathfrak g_1$ is a morphism of Lie superalgebras.

{\it Functor inverse $\iota'$}. Above we explained the meaning of the functor  inverse $\iota$ for supermanifolds. Now we show that an analogue of this functor can be defined in the category of Lie superalgebras. Let $\mathfrak g$ be a Lie superalgebra.  Denote by $\p:=\iota' \circ\pi' \circ \mathrm{gr}' \circ \oT'(\mathfrak g)$ a $\mathbb Z$-graded subsuperspace of  $\mathfrak h:=\pi'(\mathrm{gr}' (\oT'(\mathfrak g)))$ with support $\{0,1,2\}$, where $|1|=\bar 1$, given by
\begin{align*}
\p_0:= \mathfrak h_{\bar 0}, \quad \p_1= \{ Y+\mathrm{d}Y\,\,|\,\, Y\in \mathfrak g_{\bar 1}\}\subset \mathfrak h_{\alpha}\oplus\mathfrak h_{\beta}, \quad \p_2:= \mathfrak h_{\alpha+\beta}. 
\end{align*}

\begin{proposition}\label{prop tilde g is Z graded}
	The superspace $\p$ is a $\mathbb Z$-graded Lie superalgebra with support $\{0,1,2\}$ and with $|1|=\bar 1$.
\end{proposition}
\begin{proof}
	The proposition follows from a direct calculation. For example let us show that $[\p_1,\p_1] \subset \p_2$.  For $Y_1,Y_2\in\mathfrak g_1$ we have
	\begin{align*}
	[Y_1+ \mathrm{d}Y_1, Y_2+ \mathrm{d}Y_2] =  &[Y_1,Y_2] + [\mathrm{d}Y_1,Y_2] + [Y_1,\mathrm{d}Y_2] + [\mathrm{d}Y_1,\mathrm{d}Y_2] = \\
	& d[Y_1,Y_2] + d[Y_1,Y_2] = 2d[Y_1,Y_2]\in \p_2.
	\end{align*}
\end{proof}
Note that the functor $\iota'$ is defined only for Lie superalgebras of the form $\pi' \circ \mathrm{gr}' \circ \oT'(\mathfrak g)$. Further, if $f:\mathfrak g\to \mathfrak g_1$ is a morphism of Lie superalgebras, then the morphism $\pi' \circ \mathrm{gr}' \circ \oT' (f)$ can be restricted to the subalgebras $\iota' \circ\pi' \circ \mathrm{gr}' \circ \oT'(\mathfrak g) \to \iota' \circ\pi' \circ \mathrm{gr}' \circ \oT'(\mathfrak g_1)$.
Summing up, we constructed the following functor
$$
\mathrm F'_2: = \iota' \circ\pi'\circ \mathrm{gr}'\circ \oT':\SLie \to \GrLie_2.
$$
This functor is an embedding of category $\SLie$ into $\GrLie_2$.

\begin{remark}\label{rem parity change}
{\bf (1)}	It is necessary to apply the functor $\pi'$ before the functor inverse. Indeed, if $Y$ and $\mathrm{d}Y$, where $Y\in \p_1$, have different parities, we will get
	\begin{equation}\label{eq pi is nessesary}
	\begin{split}
		[Y_1+\mathrm{d}Y_1, Y_2+\mathrm{d}Y_2] = &[Y_1,Y_2] + [\mathrm{d}Y_1,Y_2] + [Y_1,\mathrm{d}Y_2] +  [\mathrm{d}Y_1,\mathrm{d}Y_2] = \\
		&\mathrm{d}[Y_1,Y_2] - \mathrm{d}[Y_1,Y_2] = 0.
	\end{split}
	\end{equation}

{\bf (2)} Another observation is the following. Lie superalgebras of the form $\mathrm{gr}'(\mathfrak k)$, where $\mathfrak k$ is a Lie superalgebra possesses another parity reversion
$$
\mathfrak k_{\bar 0}\oplus \mathfrak k_{\bar 1} \longmapsto \mathfrak k_{\bar 0}\oplus \mathfrak k_{\bar 1}[\bar 1],
$$
where $[\bar 1]$ is the shift of parity be $\bar 1$. In other words we assume that all odd elements are even. Clearly $\mathfrak k_{\bar 0}\oplus \mathfrak k_{\bar 1}[\bar 1]$ is a Lie algebra (not superalgebra). In this case the same argument, see (\ref{eq pi is nessesary}), shows that the functor inverse loses  information about original Lie superalgebra.
\end{remark}

\subsection{Donagi and Witten construction for Lie supergroups}
In this section we develop the construction of Section \ref{sec Witten} for Lie supergroups. In details, we apply the functor  $\mathrm F_2$ to a Lie supergroup $\mathcal G$. Again $\mathrm F_2$ is a composition of four functors: $\oT$ ({\it tangent}), $\mathrm{gr}$ ({\it split}), $\pi$ ({\it parity change}) and $\iota$ ({\it inverse}).

{\it Tangent functor $\oT$.} Let $\mathcal G$ be a Lie supergroup with multiplication $\mu$, inversion $\kappa$ and identity $e$. Clearly, $\oT(\mathcal G)$  is a Lie supergroup {\newMR as the functor $\oT$ preserves products}. Indeed, since $\oT$ is a functor the morthisms $\oT(\mu)$, $\oT(\kappa)$ and $\oT(e)$ satisfy the Lie supergroup axioms.

\begin{proposition}\label{prop T(g)=Lie T(G)}
	The Lie superalgebra of the Lie supergroup $\oT(\mathcal G)$
 is equal to $\oT' (\mathfrak g)$, where $\mathfrak g = \mathrm{Lie}(\mathcal G)$.
\end{proposition}

\begin{proof}
	 We shall follow the definition of the Lie functor of a Lie supergroup given in \cite{BLMS}. 
	 Let $X_1, X_2\in \T_e \mcG$ be homogeneous. Any homogeneous vector $X\in \T_e\mcG$ can be represented by a curve $g:\mathcal{R} \to \mcG$ where $\mathcal{R}$ is $\R^{1|0}$ or $\R^{0|1}$ depending on the parity of $X$, so that
$$
    g^\ast(f) = f(e) + t X(f) \bmod I_t^2, 
$$ 
for any $f\in \mcO_{\mcG}(U)$ defined in a neighbourhood $U\subset \mcG_0$ of $e$,
where $t$ is the distinguished coordinate on $\mathcal{R}$ of parity $\bar{0}$ or $\bar{1}$, and $I_t=\langle t\rangle$ is the ideal generated by $t$.
Say  $X_1, X_2\in \T_e\mcG$
are represented by curves $g_i: \mathcal{R}_i \to \mcG$, where $i=1, 2$ and $s, t$ be the distinguished coordinates on $\mathcal{\R}_1$, $\mathcal{\R}_2$, respectively. Then there exist $Z\in \T_e \mcG$ such that  for
$
    \Phi := g_1 g_2 (g_1)^{-1} (g_2)^{-1}: \mathcal{R}_1 \times \mathcal{R}_2 \to \mcG
$
we have
$$
    \Phi^\ast(f) = f(e) + (-1)^{|X| |Y|} ts Z(f) \,\bmod{I_t^2+I_s^2}.
$$
The bracket $[X_1, X_2]$ is  defined as $Z$.

We are ready to describe the Lie superalgebra $\Lie (\oT G)$. Let $(U, x^A)$ be a chart around $e\in U \subset \mcG_0$ with local coordinates $(x^A)$. Without loss of generality we may assume that $x^A(e)=0$. The group law assures that the Taylor expansion of the multiplication map
$m: \mcG\times \mcG\to \mcG$ has the form
$$
    m((x^A), (\und{x}^A)) = (x^A + \und{x}^A + \frac12 Q^A_{BC} x^C \und{x}^B \bmod I),
$$
where $I = \langle x^2, \und{x}^2\rangle$. Moreover, $Q^A_{BC} = - (-1)^{\tilde{B}\tilde{C}} Q^A_{BC}$. The inverse $(x^A)^{-1}$ is given by $(-x^A) (\bmod I)$. The multiplication on $\oT (\mcG)$ is given by $\oT (m)$, ie.
$$
    (x^A, \dot{x}^A) \cdot (\und{x}^A, \dot{\und{x}}^A) = \left(x^A + \und{x}^A + \frac12 Q^A_{BC} x^C \und{x}^B, \dot{x}^A + \dot{\und{x}}^A + \frac12 Q^A_{BC} (\dot{x}^C \und{x}^B + x^C \dot{\und{x}}^B)\right).
$$
The inverse of $(x^A, \dot{x}^A)$ is $(-x^A, -\dot{x}^A)$  modulo $I$, hence
\begin{align*}
    &(x^A, \dot{x}^A)\cdot (\und{x}^A, \und{\dot{x}}^A) \cdot (-x^A, -\dot{x}^A)\cdot (-\und{x}^A, -\und{\dot{x}}^A) \equiv \\
    &\left(x^A + \und{x}^A + \frac12 Q^A_{BC} x^C \und{x}^B, \dot{x}^A + \dot{\und{x}}^A + \frac12 Q^A_{BC} (\dot{x}^C \und{x}^B + x^C \dot{\und{x}}^B\right) \cdot \\
    &\cdot \left(-x^A - \und{x}^A + \frac12 Q^A_{BC} x^C \und{x}^B, -\dot{x}^A - \dot{\und{x}}^A + \frac12 Q^A_{BC} (\dot{x}^C \und{x}^B + x^C \dot{\und{x}}^B)\right) \\
    &\equiv (Q^A_{BC} x^C \und{x}^B, Q^A_{BC} (\dot{x}^C \und{x}^B + x^C \und{\dot{x}}^B))\, (\bmod \, I). 
\end{align*}
We read off from the last formula that $[\pa_{x^C}, \pa_{x^B}] = Q^A_{BC} \pa_{x^A}$, $[\pa_{\dot{x}^C}, \pa_{x^B}] = Q^A_{BC} \pa_{\dot{x}^A}$, and $[\pa_{\dot{x}^A}, \pa_{\dot{x}^B}] =  0$. This completes the proof.
\end{proof}

The following theorem was proved \cite{Vish_funk}.

\begin{theorem}\cite[Theorem 3]{Vish_funk}\label{theor split supergroup}
	Let $\mathcal G$ be a Lie supergroup corresponding to the super Harish-Chandra pair $(\mathcal G_0, \mathfrak g)$, where $\mathcal G_0$ is the underlying space of $\mathcal G$ and $\mathfrak g = \mathrm{Lie} (\mathcal G)$. Then  $\mathrm{gr} (\mathcal G)$ is a Lie supergroup corresponding to the following super Harish-Chandra pair $(\mathcal G_0, \mathfrak g')$, where $\mathfrak g'=\mathrm{gr}' (\mathfrak g)$.
\end{theorem}

Proposition \ref{prop T(g)=Lie T(G)} and Theorem \ref{theor split supergroup} imply the following corollary.

\begin{corollary}
	Let $\mathcal G$ be a Lie supergroup, then
	$$
	\mathrm{gr}' \circ \oT'(\Lie \mcG) = \mathrm{Lie} ( \mathrm{gr} \circ \oT(\mathcal G)).
	$$
\end{corollary}

As a consequence we get that the Lie supergroup $\gr\circ \oT(\mathcal G)$ is  a graded Lie super\-group of type $\{ 0,\alpha, \beta, \alpha+ \beta \}$, where $|\alpha|=\bar 1$ and $|\beta|=\bar 0$, with the graded Harish-Chandra pair $(\mathcal G_0, \mathrm{gr}' \circ \oT'(\mathfrak g))$ of the same type $\{ 0,\alpha, \beta, \alpha+ \beta \}$. (Compare also with Proposition \ref{prop gr T is n-fold}.)

{\it Functor parity change $\pi$}. Let us describe the functor $\pi$ using graded Harish-Chandra pairs. Let $(\mathcal G_0, \mathrm{gr}'\circ \oT'(\mathfrak g))$ be the graded Harish-Chandra pair of type $\{ 0,\alpha, \beta, \alpha+ \beta \}$ of the Lie  supergroup $\mathrm{gr}\circ \oT(\mathcal G)$ as above. In this case $\mathrm{gr}\circ \oT(\mathcal G)$ is a group object in the category of double vector bundles.  As we have seen in Section \ref{sec Witten}, the double vector bundle $\mathrm{gr}\circ \oT(\mathcal G)$ possesses the following parity reversion: we change the parity of the weight $\beta$ from even to odd. We denote this new vector bundle by $\pi\circ \mathrm{gr}\circ \oT(\mathcal G)$. Further we can define a Lie supergroup structure on $\pi\circ \mathrm{gr}\circ \oT(\mathcal G)$ using the graded Harish-Chandra pair $(\mathcal G_0,\pi'\circ  \mathrm{gr}'\circ \oT(\mathfrak g))$ of type $\Delta=\{ 0,\alpha, \beta, \alpha+ \beta \}$, where $|\alpha|=|\beta|=\bar 1$. Now the Lie supergroup  $\pi\circ \mathrm{gr}\circ \oT'(\mathcal G)$ is defined.

Let us describe the Lie supergroup morphisms of $\pi\circ \mathrm{gr}\circ \oT'(\mathcal G)$ using the language of double vector bundles. If $\mu$, $\kappa$ and $e$ are group morphisms of the Lie supergroup $\mcG$, then   $\mathrm{gr}\circ \oT(\mathcal G)$ is a group object in the category of double vector bundles with the structure morphisms $\mathrm{gr}\circ \oT (\mu)$, $\mathrm{gr}\circ \oT (\kappa)$ and $\mathrm{gr}\circ \oT (e)$. Since the category of double vector bundles possesses the parity change: $\beta$ even to $\beta$ odd, we denote by  $\pi\circ \mathrm{gr}\circ \oT(\mathcal G)$ with $\pi\circ\mathrm{gr}\circ \oT (\mu)$, $\pi\circ\mathrm{gr}\circ \oT (\kappa)$ and $\pi\circ\mathrm{gr}\circ \oT (e)$ the result of this parity change. Formulas (\ref{eq umnozh in supergroup}) tells us that this definition coincides with the definition in terms of graded Harish-Chandra pairs.

{\it Functor inverse $\iota$}. In Proposition \ref{prop tilde g is Z graded} we saw that  the Lie superalgebra $\pi'\circ  \mathrm{gr}'\circ \oT'(\mathfrak g)$ possesses a $\mathbb Z$-graded Lie subsuperalgebra $\p = \iota' \circ \pi'\circ  \mathrm{gr}'\circ \oT'(\mathfrak g)$. We define by $\iota\circ\pi\circ \mathrm{gr}\circ \oT(\mathcal G)$
the corresponding Lie subsupergroup. More precisely, we define $\iota\circ\pi\circ \mathrm{gr}\circ \oT(\mathcal G)$ using the graded Harish-Chandra pairs $
 (\mathcal G_0, \iota' \circ \pi'\circ  \mathrm{gr}'\circ \oT'(\mathfrak g)).$

\section{Generalized Donagi--Witten construction for\\ Lie superalgebras}

In Section \ref{sec gen Don Witten} we constructed an injective functor $\mathrm F:=\F_{\infty}$ from the category of supermanifolds to the category of graded manifolds. In this section we show that this functor can be defined in the category of Lie supergroups.  We start with  Lie superalgebras. In more details we will construct a functor $\mathrm F'$ from the category of Lie superalgebras $\SLie$ to the category of non-negatively $\mathbb Z$-graded Lie algebras $\GrLie_{\infty}$.  Further we will use these results for the category of Lie supergroups. Again the functor $\mathrm F'$ is a composition of four functors: the iterated antitangent functor $\oT'^{\infty}$, the functor split $\mathrm{gr}'$, the functor parity change $\pi'$ and the functor inverse $\iota'$.

\subsection{Iterated tangent functor $\oT'^{\infty}$}
Let $\mathfrak g$ be a Lie superalgebra. Let us describe the superalgebra $\oT'^{\infty}(\mathfrak g)$, where $\oT'^{\infty}:= \oT'\circ \oT' \circ\cdots $ is the infinitely many times iterated antitangent functor. Let us consider first twice iterated tangent functor $\mathrm{T'}^2(\mathfrak g)$. By definition
$$
\mathrm{T'}^2(\mathfrak g)=  \oT'(\oT'(\mathfrak g)) = \oT'(\mathfrak g\oplus \mathrm{d}_1(\mathfrak g)) = \mathfrak g\oplus \mathrm{d}_1(\mathfrak g) \oplus \mathrm{d}_2(\mathfrak g\oplus \mathrm{d}_1(\mathfrak g)).
$$
We replaced $\mathrm{d}$ from Section \ref{sec DW for superalgebras} by $\mathrm{d}_1$ and $\mathrm{d}_2$ is the second de Rham differential. The multiplication in $\mathrm{T'\circ T'}(\mathfrak g)$ is defined in a natural way. Moreover we can easily verify the following lemma.

\begin{lemma}\label{lem mult in T^2(g)}
	To obtain the multiplication in $\mathrm{T'}^2(\mathfrak g)$ we can use the following rule:
	 $[\mathrm{d}_i X,\mathrm{d}_j Y]= (-1)^{|X|} \mathrm{d}_i\mathrm{d}_j([X,Y])$ for any $X,Y\in \mathrm{T'}^2(\mathfrak g)$.
\end{lemma}

\begin{corollary}\label{cor T^2}
Since the operators $\mathrm{d}_i$ assumed to be odd, we have	$$
\mathrm{d}_i \mathrm{d}_i ([X,Y]) = 0
$$ for any $X,Y\in \mathrm{T'}^2(\mathfrak g)$.
\end{corollary}

To define the functor $\mathrm{T'}^{\infty}$  we use de Rham differentials $\mathrm{d}_1, \mathrm{d}_2, \ldots, \mathrm{d}_n, \ldots. $ Lemma \ref{lem mult in T^2(g)} and Corrolary \ref{cor T^2} imply the following lemma.

\begin{lemma}\label{lemTinfty g}
	The Lie superalgebra $\mathrm{T'}^{\infty}(\mathfrak g)$ is an infinite dimensional Lie super\-algebra with the underlying vector space
$$
\bigoplus_{p\geq 0}	\bigoplus_{i_1<\cdots <i_p} \mathrm{d}_{i_1}\cdots \mathrm{d}_{i_p}(\mathfrak g)
$$
and multiplication defined by the following formula
$$
[\mathrm{d}_i X,\mathrm{d}_j Y]= (-1)^{|X|} \mathrm{d}_i\mathrm{d}_j([X,Y])
$$
for any $X,Y\in \mathrm{T'}^{\infty}(\mathfrak g)$.
\end{lemma}

\subsection{A connection with the functor of points for Lie superalgebras}

For details about the functor of points for Lie superalgebras we refer for example to \cite[Section 2.2.4]{Gav}. Let us recall this construction. To any Lie superalgebra $\mathfrak g$ we can associate a functor $L_{\mathfrak g}$ from the category of supercommutative algebras to the category of Lie algebras. It is defined as follows
$$
L_{\mathfrak g}(A)  = (A\otimes \g)_{\bar 0},
$$
where $A$ is a super-commutative algebra. The product $A\otimes \g$ is a Lie superalgebra with the following multiplication
\begin{equation}\label{e:g_times_A}
[a\otimes X, a'\otimes X']:= (-1)^{|X| |a'|} aa'\otimes [X,X'].
\end{equation}

Comparing with Lemma \ref{lemTinfty g} we see that the Lie superalgebra $\mathrm{T'}^k(\mathfrak g)$ is isomorphic to the Lie superalgebra $\bigwedge(\xi_1,\ldots, \xi_k)\otimes\mathfrak g$. Clearly, $\mathrm{T'}^{\infty}(\mathfrak g)$ is isomorphic to $\bigwedge(\xi_1,\xi_2,\ldots)\otimes \mathfrak g $, where $\bigwedge(\xi_1,\xi_2,\ldots)$ is the Grassmann algebra with infinitely many variables $\xi_1,\xi_2,\ldots$. Now we see that $$
L_{\mathfrak g}( \bigwedge(\xi_1,\xi_2,\ldots)) = \mathrm{T'}^{\infty}(\mathfrak g)_{\bar 0}.
$$

\subsection{Functors split $\mathrm{gr}'$ and parity change $\pi'$} The functor split is defined as above.

\begin{remark}\label{rem pi gr T(g) is Delra graded Lie superalgebra}
	The Lie superalgebra $\mathrm{gr}'\circ \mathrm{T'}^{\infty}(\mathfrak g)$ is a graded Lie superalgebra with support $\Delta$, where $\Delta$ is the maximal multiplicity free weight system generated by $\alpha, \beta_1,\ldots, \beta_k,\ldots$. Here  $\alpha$ is odd, while $\beta_1,\ldots, \beta_k,\ldots$ are even. The grading is defined as follows. 	We assign the weight $\alpha+ \beta_{i_i} +\cdots+ \beta_{i_p}$ to $\mathrm{d}_{i_1}\cdots \mathrm{d}_{i_p}(Z)$, if $\mathrm{d}_{i_1}\cdots \mathrm{d}_{i_p}(Z)$ is odd, and we assign the weight $\beta_{i_i} +\cdots+ \beta_{i_p}$ to $\mathrm{d}_{i_1}\cdots \mathrm{d}_{i_p}(Z)$ if $\mathrm{d}_{i_1}\cdots \mathrm{d}_{i_p}(Z)$ is even. 
	
	For example elements of the subspace $\mathfrak g_{\bar 0}$ have weight $0$ and elements of $\mathfrak g_{\bar 1}$ have weight $\alpha$.

\end{remark}

Let us define the parity change functor $\pi'$. Let us take the Lie superalgebra $\mathrm{gr}'(\mathrm{T'}^{\infty}(\mathfrak g))$, where $\mathfrak g$ is a Lie superalgebra. The functor parity change $\pi'$ is defined as follows
$$
\mathfrak h:= \pi'(\mathrm{gr'}(\mathrm{T'}^{\infty}(\mathfrak g))) = \mathrm{gr'}(\mathrm{T'}^{\infty}(\mathfrak g))
$$
as $\mathfrak g_{\bar 0}$-modules. Further we assume that all operators $\mathrm{d}_i$ are even and again $\mathrm{d}_i\circ \mathrm{d}_i=0$. In other words, this means that
\begin{align*}
\mathfrak h_{\bar 0}: = \bigoplus_{p\geq 0}	\bigoplus_{i_1<\cdots <i_p} \mathrm{d}_{i_1}\cdots \mathrm{d}_{i_p}(\mathfrak g_{\bar 0}), \quad \mathfrak h_{\bar 1}: = \bigoplus_{p\geq 0}	\bigoplus_{i_1<\cdots <i_p} \mathrm{d}_{i_1}\cdots \mathrm{d}_{i_p}(\mathfrak g_{\bar 1}).
\end{align*}
To simplify our presentation we will use the following notations
\begin{align*}
I=\{i_1,\ldots, i_p\}, \quad J=\{j_1,\ldots, j_q\}, \quad K=\{k_1,\ldots, k_r\},
\end{align*}
and $\mathrm{d}_I$ for $\mathrm{d}_{i_1}\cdots \mathrm{d}_{i_p}$. We denote by $\sharp I:=C(I)\,\, mod\,\, 2 = \bar p$ the cardinality of $I$ modulo $2$. The multiplication in $\mathfrak h= \pi'(\mathrm{gr}'(\mathrm{T'}^{\infty}(\mathfrak g)))$ is defined by the following  rules
\begin{enumerate}
	\item[(Rule 1)] If $I\cap J\ne \emptyset$, we have $[\mathrm{d}_I(\mathfrak g), \mathrm{d}_{ J}(\mathfrak g) ] =\{0\}$.
	\item[(Rule 2)] If $\sharp I+\bar i$ and  $\sharp J+\bar j$ are odd, we have $[\mathrm{d}_I(\mathfrak g_{\bar i}), \mathrm{d}_I(\mathfrak g_{\bar j}) ] =\{0\}$.
	\item[(Rule 3)] In other cases we have 	$[\mathrm{d}_I(\mathfrak g_{\bar i}), \mathrm{d}_I(\mathfrak g_{\bar j}) ] = \mathrm{d}_{I\cup J}([\mathfrak g_{\bar i}, \mathfrak g_{\bar j}])$.
\end{enumerate}

\begin{theorem}\label{theor h is a Lie superalgebra}
	The superspace  $\mathfrak h= \pi'(\mathrm{gr}'(\mathrm{T'}^{\infty}(\mathfrak g)))$ is a Lie superalgebra.
\end{theorem}

\begin{proof} 
{\bf Step 1.}	Let us prove first that $\mathfrak h_{\bar 0}$ is a Lie algebra. Consider the following subspace
	$$
	\mathfrak r:= \bigoplus_{p\geq 0}	\bigoplus_{i_1<\cdots <i_p} \mathrm{d}_{i_1}\cdots \mathrm{d}_{i_p}(\mathfrak g_{\bar 0})\subset \mathrm{gr}'(\mathrm{T'}^{\infty}(\mathfrak g)).
	$$
 Clearly, $\mathfrak r$ is a Lie subsuperalgebra in $\mathrm{gr'}(\mathrm{T'}^{\infty}(\mathfrak g))$ and this Lie superalgebra is split. This is $\mathfrak r=\mathfrak r_{\bar 0} \oplus \mathfrak r_{\bar 1}$ with $[\mathfrak r_{\bar 1}, \mathfrak r_{\bar 1}]=\{0\}$. Such Lie superalgebras possesses a parity reversion, see Remark \ref{rem parity change} part (2). We can see that in notations of Remark \ref{rem parity change} we have $\mathfrak r_{\bar 0} \oplus \mathfrak r_{\bar 1}[\bar 1] = \mathfrak h_{\bar 0}$.

{\bf Step 2.}	Let us prove that $\mathfrak h_{\bar 1}$ is an $\mathfrak h_{\bar 0}$-module. Let us take $\mathrm{d}_I(X)\in \mathrm{d}_I(\mathfrak g_{\bar 0})$, $\mathrm{d}_J(Y)\in \mathrm{d}_J(\mathfrak g_{\bar 0})$ and $\mathrm{d}_K(Z)\in \mathrm{d}_K(\mathfrak g_{\bar 1})$ with $I\cap J=\emptyset$, $I\cap K=\emptyset$ and $J\cap K=\emptyset$. We consider the following cases
\begin{enumerate}
	\item Let $\sharp I=\sharp J=\bar 0$, any $K$. Then
	\begin{align*}
	[\mathrm{d}_I(X),[\mathrm{d}_J(Y), \mathrm{d}_K(Z)] ] + 	[\mathrm{d}_J(Y),[\mathrm{d}_K(Z), \mathrm{d}_I(X)] ]+ 	[\mathrm{d}_K(Z),[\mathrm{d}_I(X), \mathrm{d}_J(Y)] ]=\\
	\mathrm{d}_I\circ \mathrm{d}_J\circ \mathrm{d}_K([X,[Y,Z]]) + \mathrm{d}_J\circ \mathrm{d}_K\circ \mathrm{d}_I([Y,[Z,X]]) +\mathrm{d}_K\circ \mathrm{d}_I\circ \mathrm{d}_J([Z,[X,Y]]) =\\
	\mathrm{d}_I\circ \mathrm{d}_J\circ \mathrm{d}_K ([X,[Y,Z]] + [Y,[Z,X]]+ [Z,[X,Y]])=0.
	\end{align*}
	\item 	Let $\sharp I=\bar 1$, $\sharp J=\bar 0$ and $\sharp K=\bar 0$. Then
	\begin{align*}
	[\mathrm{d}_I(X),[\mathrm{d}_J(Y), \mathrm{d}_K(Z)] ] = 	[\mathrm{d}_J(Y),[\mathrm{d}_K(Z), \mathrm{d}_I(X)] ]= 	[\mathrm{d}_K(Z),[\mathrm{d}_I(X), \mathrm{d}_J(Y)] ]=0.
	\end{align*}
	\item 	Let $\sharp I=\bar 1$, $\sharp J=\bar 0$ and $\sharp K=\bar 1$. Then 	
	\begin{align*}
	[\mathrm{d}_I(X),[\mathrm{d}_J(Y), \mathrm{d}_K(Z)] ] + 	[\mathrm{d}_J(Y),[\mathrm{d}_K(Z), \mathrm{d}_I(X)] ]+ 	[\mathrm{d}_K(Z),[\mathrm{d}_I(X), \mathrm{d}_J(Y)] ]=\\\mathrm{d}_I\circ \mathrm{d}_J\circ \mathrm{d}_K ([X,[Y,Z]] + [Y,[Z,X]]+ [Z,[X,Y]])=0.
	\end{align*}
	\item 	Let $\sharp I=\sharp J=\bar 1$ and $\sharp K=\bar 0$. Then
	\begin{align*}
	[\mathrm{d}_I(X),[\mathrm{d}_J(Y), \mathrm{d}_K(Z)] ] = 	[\mathrm{d}_J(Y),[\mathrm{d}_K(Z), \mathrm{d}_I(X)] ]= 	[\mathrm{d}_K(Z),[\mathrm{d}_I(X), \mathrm{d}_J(Y)] ]=0.
	\end{align*}
		\item 	Let $\sharp I=\sharp J=\sharp K=\bar 1$. Then
	\begin{align*}
	[\mathrm{d}_I(X),[\mathrm{d}_J(Y), \mathrm{d}_K(Z)] ] = 	[\mathrm{d}_J(Y),[\mathrm{d}_K(Z), \mathrm{d}_I(X)] ]= 	[\mathrm{d}_K(Z),[\mathrm{d}_I(X), \mathrm{d}_J(Y)] ]=0.
	\end{align*}
\end{enumerate}

{\bf Step 3.} Let us check Jacobi identity for elements from $\mathfrak h_{\bar 1}$. Let us take $\mathrm{d}_I(X)\in \mathrm{d}_I(\mathfrak g_{\bar 1})$, $\mathrm{d}_J(Y)\in \mathrm{d}_J(\mathfrak g_{\bar 1})$ and $\mathrm{d}_K(Z)\in \mathrm{d}_K(\mathfrak g_{\bar 1})$ with $I\cap J=\emptyset$, $I\cap K=\emptyset$ and $J\cap K=\emptyset$ and consider the following cases
\begin{enumerate}
	\item Let $\sharp I=\sharp J=\bar 0$, any $K$. Then
	\begin{align*}
	[\mathrm{d}_I(X),[\mathrm{d}_J(Y), \mathrm{d}_K(Z)] ] = 	[\mathrm{d}_J(Y),[\mathrm{d}_K(Z), \mathrm{d}_I(X)] ]= 	[\mathrm{d}_K(Z),[\mathrm{d}_I(X), \mathrm{d}_J(Y)] ]=0.
	\end{align*}
	
\item 	Let $\sharp I=\bar 0$, $\sharp J=\bar 1$ and $\sharp K=\bar 1$. Then
\begin{align*}
[\mathrm{d}_I(X),[\mathrm{d}_J(Y), \mathrm{d}_K(Z)] ] + 	[\mathrm{d}_J(Y),[\mathrm{d}_K(Z), \mathrm{d}_I(X)] ]+ 	[\mathrm{d}_K(Z),[\mathrm{d}_I(X), \mathrm{d}_J(Y)] ]=\\\mathrm{d}_I\circ \mathrm{d}_J\circ \mathrm{d}_K ([X,[Y,Z]] + [Y,[Z,X]]+ [Z,[X,Y]])=0.
\end{align*}

	\item 	Let $\sharp I=\sharp J=\sharp K=\bar 1$. Then again
\begin{align*}
[\mathrm{d}_I(X),[\mathrm{d}_J(Y), \mathrm{d}_K(Z)] ] + 	[\mathrm{d}_J(Y),[\mathrm{d}_K(Z), \mathrm{d}_I(X)] ]+ 	[\mathrm{d}_K(Z),[\mathrm{d}_I(X), \mathrm{d}_J(Y)] ]=\\\mathrm{d}_I\circ \mathrm{d}_J\circ \mathrm{d}_K ([X,[Y,Z]] + [Y,[Z,X]]+ [Z,[X,Y]])=0.
\end{align*}
\end{enumerate}
The proof is complete.
\end{proof}

\begin{remark}
	It is unexpected that such a parity change can give a well-defined Lie superalgebra. 
\end{remark}

\subsection{Functor inverse $\iota'$}

Our goal now is to define a $\mathbb Z$-graded subsuperalgebra $\mathfrak p:= \iota' \circ \pi'\circ\mathrm{gr}'\circ \mathrm{T'}^{\infty}(\mathfrak g)$ in $\pi'\circ\mathrm{gr}'\circ \mathrm{T'}^{\infty}(\mathfrak g)$. We put
\begin{align*}
&\mathfrak p_0:= \mathfrak g_{\bar 0};\\
& \mathfrak p_1:= \mathrm{diag} ( \mathfrak g_{\bar 1} \oplus \bigoplus_{i} \mathrm{d}_i(\mathfrak g_{\bar 1}));\\
& \mathfrak p_2:= \mathrm{diag} (\bigoplus_{i} \mathrm{d}_i(\mathfrak g_{\bar 0}) \oplus  \bigoplus_{i<j} \mathrm{d}_i\mathrm{d}_j(\mathfrak g_{\bar 0})  );\\
& \mathfrak p_3:= \mathrm{diag}(\bigoplus_{i<j} \mathrm{d}_i\mathrm{d}_j(\mathfrak g_{\bar 1}) \oplus  \bigoplus_{i<j<k} \mathrm{d}_i\mathrm{d}_j\mathrm{d}_k(\mathfrak g_{\bar 1})  );\\
&\cdots\\
&  \mathfrak p_n:= \mathrm{diag} ( \bigoplus_{C(I)=n-1} \mathrm{d}_I(\mathfrak g_{\bar n}) \oplus  \bigoplus_{C(J)=n} \mathrm{d}_J(\mathfrak g_{\bar n}) );\\
&\cdots
\end{align*}
Here $C(I)$ is the cardinality of $I$. 
\begin{proposition}\label{prop p is Lie subalgebra}
	The subsuperspace $\mathfrak p$ is a $\mathbb Z$-graded Lie subsuperalgebra.
\end{proposition}

\begin{proof}
	 Let us prove that $[\mathfrak p_i,\mathfrak p_j ] \subset \mathfrak p_{i+j}$. Let us take $X\in \mathfrak g_{\bar i}$ and $Y\in \mathfrak g_{\bar j}$. Consider
\begin{align*}
	 \left[ \sum_{C(I)=i-1} \mathrm{d}_I(X) +   \sum_{C( J)=i} \mathrm{d}_J(X),
	  \sum_{C (I)'=j-1} \mathrm{d}_{I'}(Y) + \sum_{C(J')=j} \mathrm{d}_{J'}(Y)
	  \right].
	 \end{align*}
	Using Rules (2) we get
	\begin{align*}
	\left[\sum_{C( I)=i-1} \mathrm{d}_I(X), \sum_{C( I)'=j-1} \mathrm{d}_{I'}(Y) \right] =0.
	\end{align*}
Further,
\begin{align*}
\left[ \sum_{C(I)=i-1} \mathrm{d}_I(X),
   \sum_{C( J')=j} \mathrm{d}_{J'}(Y)
\right] +
\left[ \sum_{C( J)=i} \mathrm{d}_J(X),
\sum_{C( I')=j-1} \mathrm{d}_{I'}(Y)
\right] +\\
  \left[ \sum_{C( J)=i} \mathrm{d}_J(X), \sum_{C( J')=j} \mathrm{d}_{J'}(Y) \right]=\\
\underbrace{\left(\binom{i+j-1}{j}+\binom{i+j-1}{j-1}\right)}_{\binom{i+j}{i}} \sum_{C (K)=i+j-1} \mathrm{d}_K([X,Y]) +\\ 
\binom{i+j}{i}\sum_{C(K)=i+j} \mathrm{d}_K([X,Y]).
\end{align*}
The proof is complete. 
\end{proof}

If we have a morphism $\psi$ of Lie superalgebras, then  $\pi'\circ \mathrm{gr}'\circ \mathrm{T'}^{\infty} (\psi)$ preserves subalgebras $\mathfrak p$'s. Therefore we can define $\iota'$ on  morphisms. Therefore the functor $\iota'$ is defined on the image of the functor $\pi'\circ \mathrm{gr}'\circ \mathrm{T'}^{\infty}$.
Summing up, we constructed the following functor
$$
\mathrm F': = \iota' \circ\pi'\circ \mathrm{gr}'\circ \mathrm{T'}^{\infty}:\SLie \to \GrLie.
$$

\begin{remark}
	The Lie superalgebra $\mathfrak p$ is ``locally isomorphic" to $\mathfrak g$ in the following sense. We have $\mathfrak p_{ 0}\simeq \mathfrak g_{\bar 0}$ as Lie superalgebras and $\mathfrak p_{n}\simeq \mathfrak g_{\bar n}$ as $ \mathfrak g_{\bar 0}$-modules for any $n$.
\end{remark}

\subsection{The functor $\F_{\infty}$ as a inverse limit}\label{rem direct lim of F'}

We can define the Lie superalgebras $\oT'^{\infty}(\g)$, $\gr'(\oT'^{\infty}(\g))$,   $\pi'(\gr'(\oT'^{\infty}(\g)))$ and  $\iota'(\pi'(\gr'(\oT'^{\infty}(\g))))$ using the inverse limit.  Indeed,
\begin{align}
\oT'^{\infty}(\g) &= \varprojlim \oT'^{n}(\g)\label{eq1};\\
 \gr'(\oT'^{\infty}(\g)) &= \varprojlim \gr'(\oT'^{n}(\g))\label{eq2};\\
 \pi'(\gr'(\oT'^{\infty}(\g))) &= \varprojlim \pi'(\gr'(\oT'^{n}(\g)))\label{eq3};\\
\iota'(\pi'(\gr'(\oT'^{\infty}(\g)))) &= \varprojlim \iota'(\pi'(\gr'(\oT'^{n}(\g))))\label{eq4}.
\end{align}

Lemma \ref{lemTinfty g} implies (\ref{eq1}). Further, Equality (\ref{eq2})  follows from definition of the functor $\gr'$.  To obtain (\ref{eq3}) we can repeat Theorem \ref{theor h is a Lie superalgebra}  assuming that we have only $n$ de Rham differentials $\dd_i$. Finally  (\ref{eq4}) follows from  Proposition \ref{prop p is Lie subalgebra} again assuming that we have only $n$ de Rham differentials $\dd_i$. Summing up,
$$
\F'(\g) = \F'_{\infty} (\g) = \varprojlim \F'_{n} (\g).
$$

\section{Generalized Donagi--Witten construction for\\ Lie supergroups}
Again the functor $\mathrm F$ is a composition of four functors: the iterated tangent functor $\oT^{\infty}$, the functor split $\mathrm{gr}$, the functor parity change $\pi$ and the functor inverse $\iota$. Let $\mathcal G$ be a Lie supergroup with the supergroup morphisms $\mu$, $\kappa$ and $e$.

\subsection{Iterated antitangent functor $\oT^{\infty}$}
Above we considered the tangent functor $\oT$ applying to a Lie supergroup $\mathcal G$. We saw that $\oT (\mathcal G)$ is a Lie supergroup again. Therefore we can iterate this procedure and get the Lie supergroup
$$
\oT^2 (\mathcal G):= \oT(\oT( \mathcal G)).
$$
By definition we put
$$
\oT^n (\mathcal G):= \oT\circ \cdots \circ
\oT( \mathcal G),
$$
where on the left hand side the tangent functor $\oT$ is iterated $n$ times. Now we define 
\begin{align*}
    &\oT^{\infty} (\mcG) = \varprojlim \oT^{n} (\mcG),\\ \oT^{\infty}(\mu) = \varprojlim \oT^{n} (\mu),\quad &\oT^{\infty}(\kappa) = \varprojlim \oT^{n} (\kappa),\quad \oT^{\infty}(e) = \varprojlim \oT^{n} (e),
\end{align*}
see Section \ref{sec inverse limit}.  
Clearly  the infinite dimensional supermanifold $\oT^{\infty} (\mcG)$ is a Lie supergroup, since the morphisms $\oT^{\infty} (\mu)$, $\oT^{\infty} (\kappa)$ and $\oT^{\infty} (e)$ satisfy the group axioms.

Let us describe $\oT^{\infty}(\mathcal G)$ in terms of graded Harish-Chandra pairs. Above we saw that
$$
\mathrm{T'}^{n}(\mathfrak g)\simeq \bigwedge(\xi_1,\xi_2,\ldots\xi_n)\otimes \mathfrak g\quad \text{and} \quad \mathrm{T'}^{\infty}(\mathfrak g)\simeq \bigwedge(\xi_1,\xi_2,\ldots)\otimes \mathfrak g,
$$
where $\bigwedge(\xi_1,\xi_2,\ldots)$ is the Grassmann algebra with infinitely many variables $\xi_1,\xi_2,\ldots$ labeled by natural numbers. We can identify the structure sheaf $\mathcal O_{\oT^{\infty}(\mathcal G)}$ of $\oT^{\infty}(\mcG)$  with
$$
\mathrm{Hom}'_{\mathcal U(\mathfrak g_{\bar 0})} (\mathcal U( \oT^{\infty}(\mathfrak g)), \mathcal F_{\mathcal G_0}) \subset  \mathrm{Hom}_{\mathcal U(\mathfrak g_{\bar 0})} (\mathcal U( \oT^{\infty}(\mathfrak g)), \mathcal F_{\mathcal G_0}),
$$
where $\mathrm{Hom}'$  are all $\mathcal U(\mathfrak g_{\bar 0})$-homomorphisms that are zero on an ideal generated by
$$
\bigoplus_{q>0} \bigwedge^q(\xi_{n},\xi_{n+1},\ldots)\otimes \mathfrak g \subset \mathcal U( \oT^{\infty}(\mathfrak g))
$$
for some $n\geq 1$.

\subsection{Functor split $\mathrm{gr}$, functor parity change $\pi$ and functor inverse $\iota$}

  The functor $\mathrm{gr}$ for $\oT^{\infty}(\mathcal G)$ is defined as above. More precisely, let $\mathcal J$ be the sheaf of ideals in $\mathcal O_{\oT^{\infty}(\mathcal G)}$ generated by odd elements. We get a filtration in $\mathcal O_{\oT^{\infty}(\mathcal G)}$ by the subsheaves $\mathcal J^p$, where $p\geq 0$. The corresponded graded sheaf we denote by $\mathrm{gr}  \mathcal O_{\oT^{\infty}(\mathcal G)}$. By definition $\mathrm{gr}  \mathcal O_{\oT^{\infty}(\mathcal G)}$ is the structure sheaf of $\mathrm{gr} (\oT^{\infty}(\mathcal G))$.
 In terms of graded Harish-Chandra pairs we have
$$
\mathcal O_{\mathrm{gr}(\oT^{\infty}(\mathcal G))} = \mathrm{Hom}'_{\mathcal U(\mathfrak g_{\bar 0})} (\mathcal U(\mathrm{gr'}\circ \mathrm{T'}^{\infty}(\mathfrak g)), \mathcal F_{\mathcal G_0}),
$$
where $\mathrm{Hom}'$  are all $\mathcal U(\mathfrak g_{\bar 0})$-homomorphisms that are  zero on some ideal generated by
$$
\mathrm{gr'}\Big(\bigoplus_{q>0} \bigwedge^q(\xi_{n},\xi_{n+1},\ldots)\otimes \mathfrak g\Big) \subset \mathcal U( \mathrm{gr'}\circ\mathrm{T'}^{\infty}(\mathfrak g)), \,\, n\geq 1.
$$
The supermanifold $\mathrm{gr}(\oT^{\infty}(\mathcal G))$ is an $\infty$-fold vector bundle of type $\Delta$, where $\Delta$ is generated by an odd weight $\alpha$ and even weights $\beta_1,\beta_2,\ldots$, see Proposition \ref{prop gr T is n-fold}. Since $\mathrm{gr'}\circ \mathrm{T'}^{\infty}(\mathfrak g))$ is a graded Lie superalgebra of type $\Delta$, the universal enveloping algebra 
$$
\mathcal U(\mathrm{gr'}\circ \mathrm{T'}^{\infty}(\mathfrak g))= \bigoplus_{\delta\in \Z\alpha\times \Z\beta_1\times \cdots} \mathcal U(\mathrm{gr'}\circ \mathrm{T'}^{\infty}(\mathfrak g))_{\delta}
$$ 
is $\Z\times\Z\times \cdots$-graded.  We have
\begin{equation}\label{eq Delta grading in O_G}
\mathcal O_{\mathrm{gr}(\oT^{\infty}(\mathcal G))} = \bigoplus_{\delta\in \Z\alpha\times \Z\beta_1\times \cdots}(\mathcal O_{\mathrm{gr}(\oT^{\infty}(\mathcal G))})_{\delta},
\end{equation}
where
$$
(\mathcal O_{\mathrm{gr}(\oT^{\infty}(\mathcal G))})_{\delta} = \mathrm{Hom}_{\mathcal U(\mathfrak g_{\bar 0})} (\mathcal U(\mathrm{gr'}\circ \mathrm{T'}^{\infty}(\mathfrak g))_{\delta}, \mathcal F_{\mathcal G_0}).
$$
Note that in this formula we can omit $'$ and write simply $\mathrm{Hom}$. We can use the equality (\ref{eq Delta grading in O_G}) as a definition of the structure sheaf $\mathcal O_{\mathrm{gr}(\oT^{\infty}(\mathcal G))}$. Since $\gr$ is a functor, the structure morphisms in $\mathrm{gr}(\oT^{\infty}(\mathcal G))$ are graded as well. In terms of inverse limit we have $\mathrm{gr}(\oT^{\infty}(\mathcal G)=  \varprojlim \mathrm{gr}(\oT^{n}(\mathcal G))$.

To define the functor $\pi$ we change parities in any $\gr\circ \oT^{n}(\mathcal G)$. In terms of graded Harish-Chandra pairs we get
$$
\mathcal O_{\pi\circ\mathrm{gr}\circ\oT^{\infty}(\mathcal G)} = \mathrm{Hom}'_{\mathcal U(\mathfrak g_{\bar 0})} (\mathcal U(\pi' \circ\mathrm{gr'} \circ\mathrm{T'}^{\infty}(\mathfrak g)), \mathcal F_{\mathcal G_0}),
$$
where $\mathrm{Hom}'$  are all $\mathcal U(\mathfrak g_{\bar 0})$-homomorphisms that are  zero on some ideal generated by
$$
\pi'\circ\mathrm{gr'}\Big(\bigoplus_{q>0} \bigwedge^q(\xi_{n},\xi_{n+1},\ldots)\otimes \mathfrak g\Big) \subset \mathcal U( \pi'\circ\mathrm{gr'}\circ\mathrm{T'}^{\infty}(\mathfrak g)),\,\,n\geq 1.
$$
In terms of inverse limit again we have $\pi\circ \mathrm{gr}(\oT^{\infty}(\mathcal G)=  \varprojlim \pi\circ\mathrm{gr}(\oT^{n}(\mathcal G))$. The Lie supergroup $\pi\circ\mathrm{gr}\circ\oT^{\infty}(\mathcal G)$ is a graded Lie supergroup of type $\pi(\Delta)$, where $\pi(\Delta)$ is the maximal multiplicity free system generated by odd weights $\alpha, \beta_1, \beta_2, \ldots$.

A similar idea we use for functor $\iota$. Recall that $\mathrm{F}:=\iota \circ\pi\circ \mathrm{gr}\circ \oT^{\infty} $ and we denote $\mathrm{F}_n:= \iota \circ\pi\circ \mathrm{gr}\circ \oT^{n-1}$.
 First of all we define the graded Lie supergroup $\mathrm{F}_n (\mathcal G)$ of degree $n$ (that is of type $\{0,1,\ldots,n\}$ with $|1|= \bar 1$) as the $\mathbb Z$-graded Lie supergroup corresponding to the $\Z$-graded Harish-Chandra pair
$ (\mathcal G_0, \mathrm{F}'_n(\mathfrak g)).$ For any $n$ we have a natural homomorphism $\mathrm{F}'_{n+1}(\mathfrak g) \to \mathrm{F}'_n(\mathfrak g).$
This induces a homomorphism of enveloping algebras
$$
\Phi_n:\mathcal U(\mathrm{F}'_{n+1}(\mathfrak g)) \to \mathcal U(\mathrm{F}'_n(\mathfrak g)).
$$
Therefore we have a natural map of sheaves 
$$
\mathcal O_{\mathrm{F}_n(\mathcal G)} = \mathrm{Hom}_{\mathcal U(\mathfrak g_{\bar 0})} (\mathcal U(\mathrm{F}'_n(\mathfrak g)), \mathcal F_{\mathcal G_0}) \hookrightarrow \mathcal O_{\mathrm{F}_{n_+1}(\mathcal G)} = \mathrm{Hom}_{\mathcal U(\mathfrak g_{\bar 0})} (\mathcal U(\mathrm{F}'_{n+1}(\mathfrak g)), \mathcal F_{\mathcal G_0})
$$
Now we can identify the structure sheaf of $\mathrm{F}(\mathcal G)$  with
$$
\mathcal O_{\mathrm{F}(\mathcal G)} = \mathrm{Hom}'_{\mathcal U(\mathfrak g_{\bar 0})} (\mathcal U(\mathrm{F}'(\mathfrak g)), \mathcal F_{\mathcal G_0}),
$$
where $\mathrm{Hom}'$ are all $\mathcal U(\mathfrak g_{\bar 0})$-homomorphisms that are  zero on some $\mathrm{Ker} (\Phi_n)$. Further we see that $\mathcal O_{\mathrm{F}(\mathcal G)}$ is $\mathbb Z$-graded. Indeed,
$$
\mathcal O_{\mathrm{F}(\mathcal G)} = \bigoplus_{n\in \mathbb Z} (\mathcal O_{\mathrm{F}(\mathcal G)})_n = \bigoplus_{n\in \mathbb Z}
\mathrm{Hom}_{\mathcal U(\mathfrak g_{\bar 0})} (\mathcal U(\mathrm{F}'(\mathfrak g))_n, \mathcal F_{\mathcal G_0}).
$$
The graded Lie supergroup morphisms can be defined by formulas (\ref{eq umnozh in supergroup}) or using inverse limit. 
In terms of inverse limit again we have $\F(\mathcal G)=  \varprojlim \F_{n}(\mathcal G))$.

\section{Coverings and semicoverings of a Lie superalgebra and\\ a Lie supergroup}

In this section we give a definition of a covering and a semicovering of a Lie superalgebra and a Lie supergroup. Further we show that the generalized Donagi--Witten construction leads to a construction of a covering and semicovering spaces of a Lie superalgebra and a Lie supergroup.  The case of any supermanifold will be considered in \cite{RV}. 

\subsection{Coverings and semicoverings of a Lie superalgebra}

We start with Lie superalgebras. Throughout this subsection we fix a surjective homomorphism $\phi:A\to B$ of abelian groups.

\subsubsection{A $\phi$-covering of a $B$-graded Lie superalgebra along a homomorphism $\phi: A\to B$}

\begin{definition}\label{def phi covering, algebras}  A $\phi$-covering of a $B$-graded Lie superalgebra $\g$ along a surjective homomorphism $\phi: A\to B$ of abelian groups is an $A$-graded superalgebra $\p = \oplus_{\alpha\in A} \p_{\alpha}$ together with a homomorphism $\Pi': \p\to \g$ such that
$\Pi'|_{\p_a}: \p_a \to \g_{\phi(a)}$ is a linear bijection for any $\alpha\in A$.
\end{definition}

 Note that the bracket on $\p$ is fully determined by the bracket on $\g$. Indeed, for $X\in \p_\alpha$, $X'\in \p_{\alpha'}$ we have $\Pi'([X, X']_{\p}) = [\Pi'(X), \Pi'(X')]_{\g} \in \g_{\phi(\alpha + \alpha')}$, hence
$$
    [X, Y]_{\p} = (\Pi'_{\alpha+\alpha'}|_{\p_{\alpha + \alpha'}})^{-1}([\Pi'(X), \Pi'(X')]_{\g}).
$$
Also $\Pi'_\alpha=\Pi'|_{\p_\alpha}: \p_\alpha\to \g_{\phi(\alpha)}$ is a $\g_0$-module map, where we identify the Lie superalgebras $\p_0$ and $\g_0$ via $\Pi_0$.

\begin{proposition}\label{p:univ_1} Let $\aaa$, $\g$ be an $A$- and a $B$-graded Lie superalgebras, respectively.
	Let $\psi:\aaa\to \g$ be a $B$-graded homomorphism of Lie superalgebras\footnote{Any $A$-graded Lie superalgebra is automatically $B$-graded.}, and let $\p$  be a covering of $\g$ along $\phi:A\to B$. Then there exists a unique $A$-graded homomorphism $\Psi: \aaa\to \p$ such that the following diagram is commutative
		$$
		\begin{tikzcd}
		 & \p  \arrow[d, "\Pi'"]\\
		\aaa \arrow[ur, dashrightarrow, "\exists ! \Psi"]   \arrow[r, "\psi"]  & \g
		\end{tikzcd}
		$$
\end{proposition}

\begin{proof}
	 We define $\Psi$ as a linear map such that $\Psi(\aaa_s)\subset \p_s$ for any $s \in A$ and such that  $\Pi'\circ \Psi = \psi$. Let us check that $\Psi$  is a homomorphism. Indeed, let us take $X\in \aaa_s$ and $Y\in \aaa_t$. Then
	\begin{align*}
	\Pi' ([\Psi(X),\Psi(Y) ]) =  [\Pi' \circ \Psi(X),\Pi' \circ\Psi(Y) ]= [\psi(X),\psi(Y)] =\\
	\psi([X,Y]) = \Pi'\circ	\Psi([X,Y]).
	\end{align*}
	Since by definition of $\Psi$ both $\Psi([X,Y])$ and $[\Psi(X),\Psi(Y)]$ are in $\p_{s+t}$ and $\Pi'$ is locally bijective, we get the equality $[\Psi(X),\Psi(Y) ]= \Psi([X,Y])$.
\end{proof}

\begin{proposition}\label{p:lift_of_homomorphism}
Let $f:\mathfrak g\to \tilde\g$ be a homomorphism of $B$-graded Lie superalgebras. Then there exists  unique homomorphism $F$ of  $\phi$-coverings $\p$ and $\tilde\p$ such that the following diagram is commutative
$$
\begin{tikzcd}
\mathfrak p \arrow[r, "F"] \arrow["\Pi'", d]
& \tilde\p \arrow[d, "\tilde\Pi'" ] \\
\mathfrak g\arrow[r,  "f" ]
& \tilde\g
\end{tikzcd}
$$
\end{proposition}
\begin{proof} 
It follows immediately from Proposition~\ref{p:univ_1}, just take $\psi = f\circ \Pi'$.
\end{proof}
From Proposition \ref{p:lift_of_homomorphism} it follows that $\phi$-coverings are unique up to isomorphism.

\subsubsection{A $\phi$-covering and $\phi$-semicovering with support $C$}

Let $A$ and $B$ be abelian groups, $\phi: A\to B$ be a surjective homomorphism and $C\subset A$ be a subset. 

\begin{definition}\label{def phi covering with supprt algebras}  A $\phi$-covering with support $C$ of a $B$-graded Lie superalgebra $\g$ along a surjective homomorphism $\phi: A\to B$  is an $A$-graded superalgebra $\p = \oplus_{\alpha\in A} \p_{\alpha}$ with $supp (\p)=C$ such that $\phi(C)=B$ together with a surjective homomorphism $\Pi': \p\to \g$ such that
$\Pi'|_{\p_a}: \p_a \to \g_{\phi(a)}$ is a linear bijection for any $\alpha\in C$.
\end{definition}

Let $A$, $B$, $C$ and $\phi$ be as above, $\g$ and $\g'$ be a $A$-graded and $B$-graded Lie superalgebra, respectively. 

\begin{definition}
A map $\Psi: \g\to \g'$ is called  a partial homomorphism if $\Psi([X,Y ])= [\Psi(Y), \Psi(Y)]$ for any $X\in \g_{\alpha}$, $Y\in \g_{\beta}$ such that $\alpha$, $\beta$ and $\alpha+ \beta$ are in $C$.  
\end{definition}

\begin{definition}  A $\phi$-semicovering with support $C$ of a $B$-graded Lie superalgebra $\g$ along a surjective homomorphism $\phi: A\to B$  is an $A$-graded Lie superalgebra $\p = \oplus_{\alpha\in C} \p_{\alpha}$ with $supp(\p)=C$ such that $\phi(C)=B$ together with a surjective partial homomorphism $\Pi': \p\to \g$ such that
$\Pi'|_{\p_a}: \p_a \to \g_{\phi(a)}$ is a linear bijection for any $\alpha\in C$.
\end{definition}

For a $\phi$-covering and $\phi$-semicovering with support $C$ of a $B$-graded Lie superalgebra $\g$ we can prove analogues of Propositions \ref{p:univ_1} and \ref{p:lift_of_homomorphism}. 

\begin{proposition}\label{p:univ_1 C} {\bf (1)} Let $\aaa$, $\g$ be an $A$- and $B$-graded Lie superalgebras, respectively, and $supp (\aaa)=C$. 	Let $\psi:\aaa\to \g$ be a $B$-graded homomorphism of Lie superalgebras, and let $\p$  be a $\phi$-covering (or $\phi$-semicovering) with support $C$ of $\g$. Then there exists a unique $A$-graded homomorphism $\Psi: \aaa\to \p$ such that
	$ \psi= \Pi'\circ \Psi$.

{\bf (2)} Let $f:\mathfrak g\to \tilde\g$ be a homomorphism of $B$-graded Lie superalgebras. Then there exists  unique homomorphism $F$ of  $\phi$-coverings (or $\phi$-semicoverings) $\p$ and $\tilde\p$ with support $C$  such that $f \circ\Pi' = \tilde{\Pi}'\circ F$. 
\end{proposition}

\begin{proof}  We define $\Psi$ as a linear map such that $\Psi(\aaa_s)\subset \p_s$ for any $s \in C$ and such that  $\Pi'\circ \Psi = \psi$. Now we just repeat arguments of the proofs of Propositions \ref{p:univ_1} and \ref{p:lift_of_homomorphism}. One non-trivial point is the proof that $\Psi$ is a homomorphism in the case of a semicovering. We have for $X\in \aaa_{s}$ and $Y\in \aaa_t$, where $s,t,s+t\in C$,
	\begin{align*}
	\Pi' ([\Psi(X),\Psi(Y) ]) =  [\Pi' \circ \Psi(X),\Pi' \circ\Psi(Y) ]= [\psi(X),\psi(Y)] =\\
	\psi([X,Y]) = \Pi'\circ	\Psi([X,Y]).
	\end{align*}
	Since by definition of $\Psi$ both $\Psi([X,Y])$ and $[\Psi(X),\Psi(Y)]$ are in $\mathfrak p_{s+t}$ and $\Pi'$ is locally bijective, we get the equality $[\Psi(X),\Psi(Y) ]= \Psi([X,Y])$.
In the case $s,t\in C$, but $s+t\notin C$, we have
$$
[\Psi(X),\Psi(Y) ] = \Psi([X,Y])=0.
$$
\end{proof}
From Proposition \ref{p:univ_1 C} it follows that $\phi$-coverings (and $\phi$-semicoverings) with support $C$ are unique up to isomorphism.

\subsection{A covering and a semicovering of a Lie supergroup}

Unlike the notion of a  covering of a Lie superalgebra, as far as we know, the notion of a covering of a Lie supergroup was never considered in the literature before. One possible way to give a definition of a  covering of a Lie supergroup is to use the graded covering of the corresponding Lie superalgebra. However we suggest a different way that is closer to the notion of a topological covering space and this approach can be used to define a covering for any supermanifold, see \cite{RV}.

We start with a definition of a semicovering  spaces for a Lie supergroup. In Section \ref{sec Z-covering, Lie superalgebra and Lie supergroup} we will show that for any Lie supergroup $\mcG$ the image  $\F (\mcG)$ is a $(\Z\to\Z_{\bar 2})$-covering of $\mcG$ with support $\{0,1,2,\ldots\}$.   In Section \ref{sebsec matrix form} we will give a simple explicit construction of a covering of any matrix Lie supergroup. Note that in general such a simple construction for a Lie supergroup is not applicable for the case of a supermanifold. Recall that if $\mathcal M$ is a supermanifold, we denote by $\mathcal O_{\mathcal M}$ its structure sheaf and by $\mathcal F_{\mathcal M_0}$ the structure sheaf of the underlying space $\mathcal M_0$.

\subsubsection{A semicovering of a Lie supergroup}

 Let $\mathcal P$ be a graded Lie supergroup of degree $n$ (or equivalently of type $C=\{0,1,\ldots, n\}$ with $|1|=\bar 1$) with multiplication morphism $\nu$ and $\mcG$ be a Lie supergroup with multiplication morphism $\mu$. Assume in addition that $\mathcal P_0=\mcG_0$. 
\begin{definition}
    A sum of $n$ morphisms $\Pi^n= \bigoplus\limits_{c=0}^n \Pi_c$, where
    $$
    \Pi_c=(id, \Pi^*_c): (\mathcal G_0, (\mathcal O_{\mathcal P})_c) \to (\mathcal G_0, \mathcal O_{\mathcal G}), \,\, c\in C,
    $$
    and $\Pi^*_c$ is a morphism of sheaves of vector spaces, is called a partial homomorphism of $\mathcal P$ to $\mcG$ with support $C$ if 
    $$
    \Pi^*_c(f\cdot g) = \sum_{c_1+c_2=s} \Pi^*_{c_1} (f) \cdot \Pi^*_{c_2}(g),\quad f,g\in \mathcal O_{\mcG},\,\, c, c_i\in C,
    $$
    and
    $$
    \nu^* \circ \Pi^*_c = \sum_{c_1+c_2=c} (\Pi^*_{c_1}\times \Pi^*_{c_2})\circ \mu^* \quad \text{for any $c, c_i\in C$.}
    $$
\end{definition}

Sometimes we will write a partial homomorphism $\Pi^n$ of $\mathcal P$ to $\mcG$ in the form $\Pi^n: \mathcal P\to \mcG$. Note that the sheaf morphism $(\Pi^n)^*: \mathcal O_{\mcG} \to \mathcal O_{\mathcal P}$ is in general not defined.

\begin{definition}\label{def semicovering of a Lie supergroup}  A $\phi: \Z\to \Z_2$-semicovering with support $C=\{0,1,2,\ldots,n\}$ of a Lie supergroup $\mathcal G$ is a graded Lie supergroup $\mathcal P$ of degree $n$  with $\mathcal P_0=\mathcal G_0$  together with a partial homomorphism $\Pi^n: \mathcal P\to \mcG$  such that we can choose atlases $\{\mathcal U_i\}$ and $\{\mathcal V_i\}$ on $\mcG$ and $\mathcal P$, respectively, with the same base space $(\mathcal U_i)_0= (\mathcal V_i)_0$,  with even and odd coordinates $(x_a,\xi_b)$ in $\mathcal U_i$ and  with graded coordinates $(y_a^s,\eta^t_b)$, where $s\in C$ is an even integer and  $t\in C$ is an odd integer, in $\mathcal V_i$ such that
$$
\Pi^*_s(x_a) =  y_a^s, \quad \Pi^*_t(\xi_b) =  \eta^t_b.
$$
\end{definition}


\subsubsection{A covering of a Lie supergroup}

Denote by $\mathcal P = \varprojlim \mathcal P_n$ the inverse limit of  graded Lie supergroups $\mathcal P_n$ of degree $n$  with the same underlying space $(\mathcal P_n)_0=  \mathcal G_0$. 

\begin{definition}\label{def new-covering of a Lie supergroup}  A $\phi: \Z\to \Z_2$-covering with support $C=\{0,1,2,\ldots,\}$ of a Lie supergroup $\mathcal G$ is a Lie supergroup $\mathcal P= \varprojlim \mathcal P_n$ together with a Lie supergroup homomorphism $\Pi =  \varprojlim(\Pi^k): \mathcal P \to \mathcal G$   such that $\Pi^k : \mathcal P_k\to \mcG$ is a semicovering  with support $C=\{0,1,2,\ldots,k\}$ for any $k\geq 0$.
\end{definition}

\begin{remark}\label{rem covering of a Lie supergroup} Definition  \ref{def new-covering of a Lie supergroup} implies that  in the case of a  $\phi: \Z\to \Z_2$-covering $\Pi: \mathcal P \to \mathcal G$ with support $C=\{0,1,2,\ldots,\}$ we can choose "atlases" $\{\mathcal U_i\}$ and $\{\mathcal V_i\}$ on $\mcG$ and $\mathcal P$, respectively, with the same base space $(\mathcal U_i)_0= (\mathcal V_i)_0$,  with even and odd coordinates $(x_a,\xi_b)$ in $\mathcal U_i$ and  with graded coordinates $(y_a^s,\eta^t_b)$, where $s$ is an even integer and  $t$ is an odd integer, in $\mathcal V_i$ such that
$$
pr_s\circ\Pi^*(x_a) =  y_a^s, \quad pr_t\circ\Pi^*(\xi_b) =  \eta^t_b,
$$
 where $pr_q: \mathcal O_{\mathcal P}\to (\mathcal O_{\mathcal P})_q$, $q\in \Z$, is the natural projection. In this case each $\mathcal V_i$ is an "infinite graded domain", which we understand as inverse limit of the corresponding finite graded domains. 
\end{remark}

Sometimes we will call a $\phi: \Z\to \Z_2$-covering with support $C=\{0,1,2,\ldots,\}$ of a Lie supergroup $\mathcal G$ simply a $\Z^{\geq 0}$-covering of $\mcG$. 

\begin{remark}
It looks more natural to give a definition of a covering space with support $\Z$.  The study of $\Z$-graded Lie supergroups (not necessary non-negatively $\Z$-graded) was announced in \cite{Kotov}. In the case of any $\Z$-graduation graded coordinates may be not nilpotent, this leads to significant difficulties, see \cite{Colovo}.  Therefore a definition of a covering space with support $\Z$ we leave for a future research. 
\end{remark}

\begin{remark}
    Our definition implies that in some sense $\mathcal P$ is locally diffeomorphic to $\mcG$. Indeed, let us choose $q\in \Z$ assuming for example that $q$ is even. Then the sheaf morphism $pr_s\circ\Pi^*$ determines a sheaf isomorphism of $S^*(x_a)|_{(\mathcal U_i)_0}\subset \mathcal F_{\mcG_0}$ and $S^*(y^s_a)|_{(\mathcal U_i)_0}$. Here $S^*(z_a)$ is the supersymmetric algebra generated by $z_a$. And similarly in the case of odd coordinates. In other words we have the following isomorphism of superdomains
    \begin{align*}
        ((\mathcal U_i)_0, S^*(x_a)) \simeq  ((\mathcal V_i)_0, S^*(y^s_a)), \quad ((\mathcal U_i)_0, S^*(\xi_b)) \simeq  ((\mathcal V_i)_0, S^*(\eta^t_b)).
    \end{align*}
\end{remark}

We say that a Lie supergroup $\mcG$ possesses an additional   non-negative $\Z$-grading if its structure sheaf possesses a $\Z$-grading $\mathcal O_{\mcG} = \bigoplus\limits_{q\geq 0} (\mathcal O_{\mcG})_q$ and all Lie supergroup morphisms are $\Z$-graded. In this paper we assume in addition that $(\mathcal O_{\mcG})_q\subset (\mathcal O_{\mcG})_{\bar q}$. Note that a Lie supergroup with a $\Z$-grading is not the same as  a graded Lie supergroup of degree $n$.  Now we can prove that our covering satisfies the same universal properties as a covering for a Lie superalgebra.

\begin{theorem}\label{theor univ_Lie supergroups} Let $\mathcal G$ be a Lie supergroup with an additional non-negative $\Z$-grading or a graded Lie supergroup of degree $n$ and $\mathcal G'$ be a Lie supergroup.	Let $\psi:\mathcal G\to \mathcal G'$ be a Lie supergroup homomorphism  and let $\mathcal P'$  be a $\Z^{\geq 0}$-covering of $\mathcal G'$. Then there exists a unique homomorphism of Lie supergroups $\Psi: \mathcal G\to \mathcal P'$, which preserves the $\Z$-gradings, such that the following diagram is commutative
		$$
		\begin{tikzcd}
		 & \mathcal P'  \arrow[d, "\Pi"]\\
		\mathcal G \arrow[ur, dashrightarrow, "\exists ! \Psi"]   \arrow[r, "\psi"]  & \mathcal G'
		\end{tikzcd}
		$$
\end{theorem}

\begin{proof} We split the proof into steps. 

{\bf Step 1.} We put $W_i:= \psi^{-1}(( \mathcal U_i)_0)$. Let us define first $(\Psi_{W_i})^*$ for any $i$ using coordinates from Definition \ref{def semicovering of a Lie supergroup}. We put 
    $$
   (\Psi_{W_i})^* (y_a^s ) := pr_{s} \circ \psi^* ( x_a) \in (\mathcal O_{\mcG})_s, \quad (\Psi_{W_i})^* (\eta_b^t ) := pr_{t} \circ \psi^* ( \xi_b) \in (\mathcal O_{\mcG})_t,
    $$
where $pr_q:\mathcal O_{\mcG} \to (\mathcal O_{\mcG})_q$. By \cite[Section 2.1.7, Theorema]{Leites} 
we defined a morphism $\Psi_{W_i}: (W_i,\mathcal O_{\mathcal P'}) \to \mathcal V_i$ of superdomains. 

{\bf Step 2.} The morphism $\Psi_{W_i}$ satisfies the following equality
\begin{equation}\label{eq Psi|W_i}
    (\Psi_{W_i})^* ( pr_s\circ \Pi^*(F) ) = pr_s\circ \psi^*(F),
\end{equation}
    where $s\in \Z$ and $F\in \mathcal O_{\mcG'}|_{(\mathcal U_i)_0}$. First of all assume that $F$ is a polynomial  in $(x_a,\xi_b)$. It is sufficient to assume that 
 $$    
 F= x_{a_1}\cdots x_{a_k} \xi_{b_1}\cdots \xi_{b_l} 
 $$ 
is a monomial. By Definition \ref{def new-covering of a Lie supergroup} we have
$$
\Pi^*( x_{a}) = \sum_{s=2q} y^s_a,\quad \Pi^*( \xi_{b}) = \sum_{t=2q+1} \eta^t_b.
$$
(Here for simplicity of notations we write infinite sums. We understand this sum as an inverse limit.) Further, we get
\begin{align*}
     pr_s\circ \Pi^*(x_{a_1}\cdots x_{a_k} \xi_{b_1}\cdots \xi_{b_l} ) = \sum_{\sum_i s_i+\sum_j t_j=s} y^{s_1}_{a_1}\cdots y^{s_k}_{a_k} \eta^{t_1}_{b_1}\cdots \eta^{t_l}_{b_l}.
\end{align*}
Therefore,
\begin{align*}
(\Psi_{W_i})^* (pr_s\circ \Pi^* (F)) &=\\
\sum_{\sum_i s_i+\sum_j t_j=s}& (\Psi_{W_i})^*(y^{s_1}_{a_1})\cdots (\Psi_{W_i})^*(y^{s_k}_{a_k}) (\Psi_{W_i})^* (\eta^{t_1}_{b_1})\cdots (\Psi_{W_i})^* (\eta^{t_l}_{b_l}).
\end{align*}
On the other hand,
\begin{align*}
    pr_s\circ \psi^*(F)& =  pr_s\circ \psi^*(x_{a_1}\cdots x_{a_k} \xi_{b_1}\cdots \xi_{b_l}) = \\
    \sum_{\sum_i s_i+\sum_j t_j=s}& pr_{s_1} \circ \psi^*(x_{a_1})\cdots pr_{s_k} \circ \psi^*(x_{a_k}) pr_{t_1} \circ \psi^*(\xi_{b_1})\cdots pr_{t_l} \circ \psi^*(\xi_{b_l}).
\end{align*}
Now the result follows from the definition of $\Psi_{W_i}$. If morphisms coincide on all polynomials, a standard argument, see \cite{Leites}, implies the result for any functions $F\in \mathcal O_{\mcG'}$. 

{\bf Step 3.} We have to show that  $\Psi_{W_i} = \Psi_{W_j}$ in $W_i\cap W_j$. Let $\tilde x_a= H(x_a,\xi_b)$ in $\mathcal U_i\cap \mathcal U_j$  and $\tilde y^s_a= pr_{s} \circ \psi^* ( x_a)$. We have by Step 2
\begin{align*}
    (\Psi_{W_j})^* (\tilde y^s_a) = (\Psi_{W_j})^* (pr_s\circ \Pi^*(\tilde x_a)) = (\Psi_{W_j})^* (pr_s\circ \Pi^*(H(x_a,\xi_b))) = \\
    pr_s\circ \psi^*(H(x_a,\xi_b)) =  (\Psi_{W_i})^*  (pr_s\circ \Pi^* (H(x_a,\xi_b))).
 \end{align*}
Now we can define the morphism $\Psi$ by $\Psi|_{W_i}:= \Psi_{W_i}$.
 
 {\bf Step 4.} Let us prove that $\Psi$ is a homomorphism of Lie supergroups. Denote by $\mu_{\mcG}$, $\mu_{\mcG'}$ and by $\mu_{\mathcal P}$ the multiplication morphism in $\mcG$, $\mcG$ and $\mathcal P$, respectively.  
 We have to show that 
 \begin{equation}\label{eq Psi is a  homom}
    (\Psi^*\times \Psi^*) \circ\mu_{\mathcal P}^* = \mu_{\mcG}^*\circ \Psi^*. 
 \end{equation}
 Clearly it is sufficient to show (\ref{eq Psi is a  homom}) only for coordinates $(y_a^s,\eta_b^t)$. Step $2$ implies that $$
  \Psi^* \circ pr_s\circ \Pi^* = pr_s\circ \psi^*.
 $$
 We have
 \begin{align*}
   \mu_{\mcG}^* \circ \Psi^* (y^s_a) =  \mu_{\mcG}^* \circ \Psi^* \circ pr_s\circ \Pi^*( x_a) = \mu_{\mcG}^* \circ pr_s\circ \psi^*( x_a) =  pr_s\circ \mu_{\mcG}^* \circ \psi^*( x_a)=\\
   pr_s\circ (\psi^*\times \psi^*)\circ  \mu_{\mcG'}^* ( x_a) = (\Psi^*\times \Psi^*) \circ pr_s\circ (\Pi^*\times \Pi^*)\circ \mu_{\mcG'}^* ( x_a)=\\
   (\Psi^*\times \Psi^*) \circ pr_s\circ  \mu_{\mcG}^* \circ \Pi^* ( x_a) = (\Psi^*\times \Psi^*) \circ \mu_{\mcG}^* \circ pr_s \circ \Pi^* ( x_a) =  (\Psi^*\times \Psi^*) \circ \mu_{\mcG}^*  ( y^s_a).
 \end{align*}
 For $\eta_b^t$ the proof is similar. The proof is complete.
\end{proof}

\begin{corollary}
    	If a $\Z^{\geq 0}$-covering of a Lie supergroup $\mcG$ exists, it is unique up to isomorphism.  
\end{corollary}

\begin{proof}
    In Theorem \ref{theor univ_Lie supergroups} we assume that $\mcG$ is finite dimensional. However the same argument as in the proof of  Theorem \ref{theor univ_Lie supergroups} implies that we can replace $\mcG$ by a $\Z^{\geq 0}$-covering of $\mcG$. Therefore we get the result. 
\end{proof}

If there exists a $\Z^{\geq 0}$-covering $\Pi:\mathcal P\to \mcG$ of a Lie supergroup $\mathcal G$, then from Definition \ref{def new-covering of a Lie supergroup} it follows that there exists  a $\phi: \Z\to \Z_2$-semicovering of $\mathcal G$ with support $C=\{0,1,2,\ldots,n\}$ for any $n$. Indeed, let $\mathcal P_n$ be the graded Lie supergroup of degree $n$ as in Definition \ref{def new-covering of a Lie supergroup}.  
Then by Definition \ref{def new-covering of a Lie supergroup}, $\Pi^n$ is a partial homomorphism. Further, for any Lie supergroups homomorphism $\psi: \mcG'\to \mcG$, where $\mcG'$ is a graded Lie supergroup of degree $n$, there exists unique Lie supergroup homomorpism (not a partial homomorphism!) $\Psi_n: \mcG'\to  \mathcal P_n$ such that $pr_s\circ\psi^*= \Psi^*\circ \Pi_s$. 

We finalize this section with the following proposition.  

\begin{proposition}
Let $\mcG$, $\tilde{\mcG}$ be Lie supergroups and $\mathcal P$, $\tilde{\mathcal P}$ be their $\Z^{\geq 0}$-coverings, respectively. Let $\psi:\mcG\to \tilde{\mcG}$ be a homomorphism of Lie supergroups. Then there exists  unique Lie supergroup homomorphism $\Psi$ of $\Z^{\geq 0}$-coverings $\mathcal P$ to $\tilde{\mathcal P}$ such that the following diagram is commutative
$$
\begin{tikzcd}
\mathcal P \arrow[r, "\Psi"] \arrow["\Pi", d]
& \tilde{\mathcal P} \arrow[d, "\Pi" ] \\
\mcG\arrow[r,  "\psi" ]
& \tilde{\mcG}
\end{tikzcd}
$$
\end{proposition}

\begin{proof}
  Locally  we put
    \begin{align*}
        \Psi^*(\tilde y^s_a) := pr_s\circ \Pi^*\circ \psi^*(\tilde x_a),\quad \Psi^*(\tilde \eta^t_b) := pr_t\circ \Pi^*\circ \psi^*(\tilde \xi_b).
    \end{align*}
The rest of the proof is similar to the proof of Theorem (\ref{theor univ_Lie supergroups}). 
\end{proof}

Later we will prove that such a covering exists for any Lie supergroup. More precisely we will show that $\mathrm{F}(\mcG)$ is a covering of a Lie supergroup $\mcG$.

\section{Existence of a $\mathbb Z^{\geq 0}$-covering of a Lie superalgebra and\\ a Lie supergroup}\label{sec Z-covering, Lie superalgebra and Lie supergroup}

In this section we will show that the $\mathrm{F}'$-image of a Lie superalgebra $\mathfrak g$ and the $\mathrm{F}$-image a Lie supergroup $\mcG$ are $\Z^{\geq0}$-coverings of $\mathfrak g$ and $\mcG$, respectively. We start with a Lie superalgebara.

\subsection{A $\mathbb Z^{\geq 0}$-covering of a Lie superalgebra}
In this section we will show that the Lie superalgebra $\mathfrak p:= \mathrm{F}'(\mathfrak g)$ is a $\mathbb Z^{\geq 0}$-covering  of $\mathfrak g$, see Definition \ref{def phi covering with supprt algebras}.
We are ready to prove the following theorem.

\begin{theorem}
For any Lie superalgebra $\mathfrak g$ the Lie superalgebra $\mathfrak p= \mathrm{F}'(\mathfrak g)$ is a $\mathbb Z^{\geq0}$-covering of $\mathfrak g$. 
\end{theorem}

\begin{proof}
	In notations of Proposition \ref{prop p is Lie subalgebra} the homomorphism $\Pi': \mathfrak p\to \mathfrak g$ is defined as follows
	$$
	\Pi'(\sum_{\sharp I=i-1} \mathrm{d}_I(X) + \sum_{\sharp J=i} \mathrm{d}_J(X)) = \frac{1}{i!} X \in \mathfrak g_{\bar i}.
	$$
	Denote $Z':= \sum\limits_{\sharp I=i-1} \mathrm{d}_I(Z) +  \sum\limits_{\sharp J=i} \mathrm{d}_J(Z)\in \mathfrak p_i$ for any $Z\in \mathfrak g_{\bar i}$.
	Using arguments from the proof of Proposition \ref{prop p is Lie subalgebra} we see that
	$$
	\Pi'([X',Y'])= \frac{1}{i!j!}[X,Y]= [\Pi'(X'), \Pi'(Y')]
	$$
	for any $X'\in \mathfrak p_i$ and $Y'\in \mathfrak p_j$. Clearly $\Pi'|_{\mathfrak p_{n}}: \mathfrak p_{n}\to \mathfrak g_{\bar n}$ is a bijective linear map for any $n\in \mathbb Z^{\geq 0}$.
	\end{proof}

In Proposition \ref{p:univ_1 C} we proved that if $f:\mathfrak g\to \mathfrak g'$ is a homomorphism of Lie algebras, then there exists unique Lie algebra homomorphism $\bar f: \mathfrak p\to \mathfrak p'$ such that $ f\circ \Pi' =  \Pi'\circ\bar f$. Now we can construct $\bar f$ explicitly. We put $\bar f= \mathrm F'(f)$. 

We conclude this section with the following observation. 

\begin{proposition}
    The Lie superalgebra $\mathrm F'_n (\g)$ is  a semicovering with support $\{0,\ldots,n\}$ of the Lie superalgebra $\g$  along the homomorphism $\mathbb Z\to \mathbb Z_2$. 
\end{proposition}

\subsection{A $\mathbb Z^{\geq 0}$-covering of a Lie supergroup}
Let $\mcG$ be a Lie supergroup with the Lie superalgebra $\mathfrak g$ and $\mathcal P$ be a graded Lie supergroup with the graded Lie superalgebra $\mathfrak p$ and with $\mathcal P_0=\mcG_0$. We put $\mathfrak p_{(1)}:= \mathfrak p_{1}\oplus \mathfrak p_{2} \oplus \cdots$.  By the Poincar\'{e}–Birkhoff–Witt  theorem we
have the following sheaf isomorphisms
\begin{equation}\label{eq sheaf is isomorphic to PBW}
\begin{split}
    \mathcal O_{\mcG} &= \mathrm{Hom}_{\mathcal U(\mathfrak g_{\bar 0})} (\mathcal U(\mathfrak g), \mathcal F_{\mcG_0}) \simeq  \mathrm{Hom}_{\mathbb C} (S^*(\mathfrak g_{\bar 1}), \mathcal F_{\mcG_0}) \simeq S^*(\mathfrak g^*_{\bar 1})\otimes_{\mathbb C} \mathcal F_{\mcG_0}, \\ 
    (\mathcal O_{\mathcal P})_q &= \mathrm{Hom}_{\mathcal U(\mathfrak p_{0})} (\mathcal U(\mathfrak p)_q, \mathcal F_{\mathcal P_0}) \simeq  \mathrm{Hom}_{\mathbb C} (S^q(\mathfrak p_{(1)}), \mathcal F_{\mathcal P_0})\simeq \\
    &\simeq (S^q(\mathfrak p_{(1)}))^*\otimes_{\mathbb C} \mathcal F_{\mcG_0},
\end{split}
\end{equation}
where $q\in \Z^{\geq 0}$. The isomorphism (\ref{eq sheaf is isomorphic to PBW}) leads to a certain choice of global odd and global graded coordinates on $\mcG$ and $\mathcal P$, respectively.  Indeed, let $(\xi_i)$ be a basis in $\mathfrak g_{\bar 1}^*$ and let  $(\zeta^q_b)$ be a basis in $\mathfrak p_{q}^*$, where $q \geq 1$.  Since $\mathfrak g_{\bar 1}^*$ is a direct term in $S^*(\mathfrak g^*_{\bar 1})$, we assume that $\xi_i\in S^*(\mathfrak g^*_{\bar 1})$, and similarly  $\zeta^q_b\in S^q(\mathfrak p_{(1)}))^*$. Then the system $(\xi_i)$ is a system of global odd coordinates on $\mcG$, and 
$(\zeta^q_b)$, where $q\geq 1$ is the degree of the coordinate $\zeta^q_b$, is a system of global graded coordinates of $\mathcal P$. We will call such coordinates of a (graded) Lie supergroup {\it standard}. Note that in general we cannot choose global coordinates of degree $0$.

We  need the following two lemmas. 
\begin{lemma}\label{lem functions of degree p}
Let us fix a graded chart $\mathcal U$ on a graded Lie supergroup $\mathcal P$ with standard graded coordinates $(\zeta^q_b)$, where $q\geq 1$. Denote by $\theta^q_c$ the basis in $\mathfrak p_{q}$ dual to the basis $(\zeta^q_b)$ of $\mathfrak p_{q}^*$. Let $f\in (\mathcal O_{\mathcal U})_p$ be a graded function of degree $p$. If $f(\theta^p_c) = 0$ for any $c$, then $f$ is a sum of products of coordinates of degree $p'<p$. 
\end{lemma}

\begin{proof}
    By definition of standard coordinates,  $f$ is  a sum of products of coordinates of degree $p'<p$ with respect to the coordinates $(\zeta^q_b)$. Therefore the same holds for any other coordinates. 
\end{proof}

\begin{lemma}\label{lem change of coord}
Let $\mathcal U$ be a graded superdomain with graded coordinates $(\theta^q_b)$, where $q\in \Z$ is the degree of a coordinate. The system $(\tilde{\theta}^q_b)$, where
\begin{align*}
    \tilde{\theta}^q_b = \theta^q_b + \sum_{q_1+\cdots q_k=q} f_{i_1\ldots i_k} \theta^{q_1}_{i_1} \cdots \theta^{q_k}_{i_k}, \quad f_{i_1\ldots i_k}= f_{i_1\ldots i_k}(\theta^{0}_{a}),
\end{align*}
is a new graded coordinate system in $\mathcal U$. 
\end{lemma}
\begin{proof}
    By induction. 
\end{proof}

Now we are ready to prove the main result of this section. 

\begin{theorem}\label{theor F(G) is a covering of G}
For any Lie supergroup $\mathcal G$ the image $\mathrm F(\mathcal G)$ is a $\mathbb Z^{\geq 0}$-covering of $\mathcal G$.
\end{theorem}
\begin{proof}
Denote $\mathcal P:= \mathrm F(\mathcal G)$. We have a natural homomorphism of Lie supergroups $\Pi: \mathcal P\to \mcG$ given by 
\begin{align*}
    &\Pi^*:\mathrm{Hom}_{\mathcal U(\mathfrak g_{\bar 0})} (\mathcal U(\mathfrak g), \mathcal F_{\mcG_0})\to  \mathrm{Hom}_{\mathcal U(\mathfrak p_{0})} (\mathcal U(\mathfrak p), \mathcal F_{\mathcal G_0}),\\
    &\Pi^*(f) (X)= f(\Pi'(X)), \quad  f\in \mathrm{Hom}_{\mathcal U(\mathfrak g_{\bar 0})} (\mathcal U(\mathfrak g), \mathcal F_{\mcG_0}),\,\, X\in \mathcal U(\mathfrak p),
\end{align*}
where $\Pi': \mathfrak p\to \mathfrak g$ is the differential of $\Pi$. (With the same letter we denote the corresponding homomorphism of the universal enveloping algebras.) Let us choose an atlas $\mathcal A$ on $\mcG$ and let us fix $\mathcal U\in \mathcal A$  with coordinates $(x_a, \xi_b)$, where $\xi_b$ are standard global odd coordinates. 

Let us choose a basis $Z_1,\ldots, Z_n$ in $\mathfrak g_{\bar 0}$ and a basis $T_1,\ldots, T_m$ in $\mathfrak g_{\bar 1}$, dual to the basis $(\xi_b)$. Let $Z^s_1,\ldots, Z^s_n$ is the corresponding basis in $\mathfrak p_{s}$, where $s$ is even,  and $T^t_1,\ldots, T^t_m$ is the corresponding basis in $\mathfrak p_{s}$, where $t$ is odd. Further let us identify $x_a$ with elements in $X_a\in \mathcal O_{\mcG}$, where
$$
X_a (1)=x_a, \quad X_a(T_j) =0,\quad   X_a(Z_j) = \tilde Z_j  (x_a)_{red}.
$$
Here $\tilde Z_j$ is the fundamental vector field on $\mcG$ corresponding to $Z_j$ and $\tilde Z_j (x_a)_{red}$  is the image of the function $\tilde Z_j (x_a)$ via the natural projection $\mathcal O_{\mcG} \to \mathcal F_{\mcG_0}$.

Let us define an atlas $\{\mathcal V_i\}$ as in Definition \ref{def new-covering of a Lie supergroup} and \ref{def semicovering of a Lie supergroup}. We put 
\begin{equation}\label{eq atlas F(G)}
    y^s_a:= pr_s\circ \Pi^*(x_a), \quad  \eta^t_b:= pr_t\circ \Pi^*(\xi_b)
\end{equation}
for $s$ even and $t$ odd. Let us check that $(x_a, y^s_a, \eta^t_b)$ are graded coordinates in $\mathcal P$. According above we may choose standard graded coordinates $(x_a,\tilde y^s_a, \tilde {\eta}^t_b)$, where $x_a\in \mathcal F_{\mcG_0}$ are as above, and
$$
\tilde y^s_i (Z^s_j)= \delta_{ij},\quad \tilde y^s_i (T^t_j)= 0,\quad  \tilde{\eta}^t_i (T^t_j)= \delta_{ij}, \quad  \tilde{\eta}^t_i (Z^s_j)= 0. 
$$ 
And we denote by the same letter $\tilde y^s_i$ the corresponding elements in $ (S^s(\mathfrak p/\mathfrak p_0))^*$ (or by $\tilde{\eta}^t_i$ the corresponding elements in $ (S^t(\mathfrak p/\mathfrak p_0))^*$).

   Let us show that  the coordinates $(x_a,\tilde y^s_a, \tilde {\eta}^t_b)$  may be expressed  via $(x_a,y^s_a, \eta^t_b)$. This will imply that the system $(x_a,y^s_a, \eta^t_b)$ is a system of local coordinates as well.  
We have
$$
\eta^t_b(T^t_j) =  \Pi^*(\xi_b)(T^t_j) = \xi_b(\Pi'(T^t_j)) = \xi_b (T_j)=\delta_{bj}.
$$
Similarly, checking other values of  $\eta^t_b(A)$ and $\tilde{\eta}^t_b(A)$, where $A\in \mathcal U(\mathfrak p)_t$, we get $\eta^t_b=\tilde{\eta}^t_b$.  Further, the systems of vector fields $(\frac{\partial}{\partial x_a})$ and $(\tilde Z_a)$ are both local bases of the sheaf of vector fields on $\mcG$. Therefore the matrix $A=(\tilde Z_i (x_j))_{red}$ is invertible. Let $B=A^{-1}$. Consider $\sum_a B^i_a {y}^s_a$. Then 
$$
\sum_a B^i_a {y}^s_a (Z^s_j)= \sum_a B^i_a \tilde Z_j(x_a) =
\delta_{ij} = \tilde{y}^s_i(Z^s_j). 
$$
By Lemma \ref{lem functions of degree p} the difference $\sum_a B^i_a {y}^s_a- \tilde{y}^s_i$ is a sum of products of coordinates of degrees smaller than $s$. By Lemma \ref{lem change of coord} and induction of $s$ we conclude that $\tilde{y}^s_i$ may be expressed via ${y}^{s'}_a$. The proof is complete. 
\end{proof}

\subsection{A semicovering of a Lie supergroup with support $\{0,1,\ldots, n\}$}

As we have seen after Definition \ref{def semicovering of a Lie supergroup}, we can construct a semicovering of a Lie supergroup $\mcG$ with support $\{0,1,\ldots, n\}$ using a $\Z^{\geq 0}$-covering of $\mcG$. The definition of the functor $\F$ implies that $\F_n(\mcG)$ is a semicovering of $\mcG$ with support $\{0,1,\ldots, n\}$. 

\begin{remark}
    The functor $\F_n$ is an embedding of the category of Lie supergroups to the category of $\Z$-graded Lie supergroups with support $\{0,1,\ldots, n\}$. Clearly we can recover the Lie supergroup structure on the supermanifold $\mcG$ using the Lie supergroup structure of $\F_n(\mcG)$. 
\end{remark}

\section{Loop superalgebras}\label{sebsec matrix form}

In this section we give an explicit matrix realization of a $\mathbb Z^{\geq 0}$-covering for a finite dimensional Lie superalgebra  and for a matrix Lie supergroup. Let $\mathfrak g$ be a Lie superalgebra and $\mathfrak p:= \F'(\mathfrak g)$. We can easily see that $\mathfrak p$ is a subalgebra of a loop algebra in the sense \cite[Definition 3.1.1]{Allison}, see also \cite{Eld}. Indeed, clearly we have
$$
\mathfrak p \simeq \bigoplus_{i\geq 0} \mathfrak g_{\bar i}\otimes t^i,\quad \mathfrak p_i \simeq \mathfrak g_{\bar i}\otimes t^i,
$$
where $t$ is a formal even variable. (We will get a definition of a loop algebra if we replace $i\geq 0$ by $i\in \Z$.)

Now consider the case $\mathfrak g = \mathfrak{gl}_{m|n}(\mathbb K)$. The Lie superalgebra $\mathfrak p = \F'(\mathfrak g)$ possesses a simple matrix realization. This implies by Ado's Theorem that we have such a matrix realization for any finite dimensional Lie superalgebra $\mathfrak{h}\subset \mathfrak{gl}_{m|n}(\mathbb K)$.  Recall that the general linear Lie superalgebra $\mathfrak{gl}_{m|n}(\mathbb K)$ contains all matrices in the following form
$$
\left(
\begin{array}{cc}
A& B \\
C& D\\
\end{array}\right),
$$
where $A$ is $m\times m$-matrix and  $D$ is $n\times n$-matrix over $\mathbb K$.  Then the Lie superalgebra $\mathfrak p$ contains all matrices in the following form
$$
  \left(
\begin{array}{ccccc}
A_1 & 0 & 0& 0& \cdots\\
C_1& D_1& 0 & 0&\cdots\\
A_2&B_1& A_1 & 0& \cdots\\
 C_2 &D_2& C_1 & D_1& \cdots\\
\cdots& \cdots&  \cdots &  \cdots& \cdots\\
\end{array}\right).
$$
Here $A_i$ are $m\times m$-matrices, $D_i$ are $n\times n$-matrices, $B_i$ are $m\times n$-matrices, $C_i$ are $n\times m$-matrices over $\mathbb K$ and so on.
 We have
$$
\mathfrak p = \bigoplus_{n\geq 0} \mathfrak p_n,\quad  \mathfrak p=  \mathfrak p_{\bar 0} \oplus  \mathfrak p_{\bar 1},
$$
where $\mathfrak p_{\bar 0}$ contains all matrices with $B_i=0$ and $C_j=0$,  $\mathfrak p_{\bar 1}$ contains all matrices with $A_i=0$ and $D_j=0$. The $\mathbb Z$-grading of $\mathfrak p$ has natural meaning:  $\mathfrak p_0$ contains all matrices with $B_i=0$, $C_i=0$ for any $i$, $A_j=0$ and $D_j=0$ for $j>1$;  $\mathfrak p_1$ contains all matrices with $A_i=0$, $D_i=0$ for any $i$, $C_j=0$ and $B_j=0$ for $j>1$.

 Let $\mcG$ be a Lie subsupergroup in the general linear Lie supergroup $\mathrm {GL}_{m|n}(\mathbb K)$. Using our results for the Lie superalgebra $\mathfrak{gl}_{m|n}(\mathbb K)$ we can explicitly construct $\F(\mathrm {GL}_{m|n}(\mathbb K)) $. Indeed, let $\mathrm {GL}_{m|n}(\mathbb K)$ has the following coordinate superdomain $\mathcal U$
$$
\left(
\begin{array}{cc}
X & \Xi\\
\mathrm H& Y\\
\end{array}\right),
$$
where $X\in \mathrm {GL}_{m}(\mathbb K)$, $Y\in \mathrm {GL}_{n}(\mathbb K)$ are matrices of even coordinates of $\mathcal U$ and $\Xi$, $\mathrm H$ are matrices of odd coordinates. Now  $\mathrm F(\mathrm {GL}_{m|n}(\mathbb K))$ has the following coordinate superdomain
$$
\left(
\begin{array}{ccccc}
X_1 & 0 & 0& 0& \cdots\\
\mathrm H_1& Y_1& 0 & 0&\cdots\\
X_2&\Xi_1& X_1 & 0& \cdots\\
\Xi_2&Y_2& \mathrm H_1 & Y_1& \cdots\\
\cdots& \cdots&  \cdots &  \cdots& \cdots\\
\end{array}\right).
$$
$X_i\in \mathrm {GL}_{m}(\mathbb K)$, $Y_j\in \mathrm {GL}_{n}(\mathbb K)$ are matrices of even coordinates of $\mathrm F(\mathrm {GL}_{m|n}(\mathbb K))$ and $\Xi_s$, $\mathrm H_t$ are matrices of odd coordinates. The Lie supergroup multiplication is given by usual matrix multiplication. A similar idea can be used to construct explicitly  $\mathrm F(\mathcal G)$ for any matrix Lie supergoup $\mathcal G$.

We can give a matrix realization of the Lie superalgebras $\F'_n(\g)$, $n\geq 2$.  The Lie superalgebra $\F'_n(\g)$ contains all matrices in the following form
$$
  \left(
\begin{array}{ccccc}
A_1 & 0 & 0& \cdots & 0\\
C_1& D_1& 0 & \cdots &0\\
A_2&B_1& A_1 &  \cdots & 0\\
\cdots& \cdots&  \cdots &  \cdots& \cdots\\
 C_{l} &D_l& C_{l-1} &  \cdots& D_1\\
\end{array}\right),\quad 
\left(
\begin{array}{ccccc}
A_1 & 0 & 0& \cdots & 0\\
C_1& D_1& 0 & \cdots &0\\
A_2&B_1& A_1 &  \cdots & 0\\
\cdots& \cdots&  \cdots &  \cdots& \cdots\\
 A_{r+1} &B_{r}& A_{r} &  \cdots& A_1\\
\end{array}\right),
$$
for $n=2l-1$ and for $n=2r$, respectively.  Similarly for matrix Lie supergroups. Note that in general these constructions are not applicable for supermanifolds. 

\small


\noindent
E.~V.: Departamento de Matem{\'a}tica, Instituto de Ci{\^e}ncias Exatas,
Universidade Federal de Minas Gerais,
Av. Ant{\^o}nio Carlos, 6627, CEP: 31270-901, Belo Horizonte,
Minas Gerais, BRAZIL, Institute of Mathematics, University of Cologne, Weyertal 86-90, 50931 Cologne, GERMANY, and Laboratory of Theoretical and Mathematical Physics, Tomsk State University,
Tomsk 634050, RUSSIA.

\noindent email: {\tt VishnyakovaE\symbol{64}googlemail.com}

M. ~Rotkiewicz: {Department of Mathematics, Warsaw University,  Warsaw, Poland.}

\noindent email: {\tt mrotkiew\symbol{64}mimuw.edu.pl}

\end{document}